\documentclass[12pt,twoside]{amsart}
\usepackage{amssymb,amsmath,amsfonts,amsthm}
\usepackage{mathrsfs}
\usepackage[X2,T1]{fontenc}
\usepackage[applemac]{inputenc}
\usepackage{cite}
\usepackage{latexsym}
\usepackage{verbatim}
\usepackage{cancel}
\usepackage[dvipsnames]{xcolor}
\newcommand{\japx}{\langle x\rangle}
\newcommand{\japxi}{\langle\xi\rangle}
\newcommand{\japxih}{\langle\xi\rangle_h}
\newcommand{\partialx}{\partial_x}
\newcommand{\partialxi}{\partial_\xi}
\newcommand{\doublepartial}{\partial_\xi^\alpha\partial_x^\beta}

\def\R{\mathbb R}
\def\C{\mathbb C}
\def\N{\mathbb N}

\newcommand\dslash{d\llap {\raisebox{.9ex}{$\scriptstyle-\!$}}}
\usepackage{colortbl}

\newcommand{\beqsn}{\arraycolsep1.5pt\begin{eqnarray*}}
\newcommand{\eeqsn}{\end{eqnarray*}\arraycolsep5pt}
\newcommand{\beqs}{\arraycolsep1.5pt\begin{eqnarray}}
\newcommand{\eeqs}{\end{eqnarray}\arraycolsep5pt}
\def\ds{\displaystyle}

\newtheorem{theorem}{Theorem}
\newtheorem{lemma}{Lemma}

\newtheorem{proposition}{Proposition}

\newtheorem{remark}{Remark}

\catcode`\@=11

\renewcommand{\section}%
   {\setcounter{equation}{0}\@startsection {section}{1}{\z@}{-3.5ex plus -1ex
  minus -.2ex}{2.3ex plus .2ex}{\Large\bf}}

\title[$p$-evolution equations with variable coefficients in Gevrey classes ]{The Cauchy problem for $p$-evolution equations with variable coefficients in Gevrey classes}

\author[A. Arias Junior]{Alexandre Arias Junior}
\address{ Department of Computing and Mathematics, Universidade de S\~ao Paulo,
	Ribeir\~ao Preto,
	Brazil}
\email{alexandre.ariasjunior@usp.br}

\author[A. Ascanelli]{Alessia Ascanelli}
\address{Dipartimento di Matematica ed Informatica\\Universit\`a di Ferrara\\
Via Machiavelli 30\\
44121 Ferrara\\
Italy}
\email{alessia.ascanelli@unife.it}

\author[M. Cappiello]{Marco Cappiello}
\address{Dipartimento di Matematica ``G. Peano'' \\Universit\`a di Torino\\
Via Carlo Alberto 10\\
10123 Torino\\
Italy}
\email{marco.cappiello@unito.it}

\author[E.C. Machado]{Eliakim Cleyton Machado}
\address{Department of Mathematics \\ Federal University of Paran\'a, Caixa Postal 19081 \\ CEP 81531-980, Curitiba, Brazil}
\email{eliakimmachado@ufpr.br}

\thanks{The first author was supported by São Paulo Research Foundation (FAPESP), grant 2022/01712-3. The second and third author were supported by the INdAM-GNAMPA projects CUP E53C22001930001 and CUP E53C23001670001. The fourth author wishes to thank for the financial support granted by Conselho de Desenvolvimento Cientifico e Tecnol\'ogico (CNPq), Brazil, 200229/2023-0, during his sandwich Ph.D. period in Turin, Italy, when part of this paper has been written.}

\begin{document}


\begin{abstract} We study the Cauchy problem for a class of linear evolution equations of arbitrary order with coefficients depending both on time and space variables. Under suitable decay assumptions on the coefficients of the lower order terms for $|x|$ large, we prove a well-posedness result in Gevrey-type spaces.
\end{abstract}

\maketitle



\section{Introduction and main results}\label{sec:introduction}
The class of {\it$p$-evolution equations with real characteristics} 
	has been introduced for the first time in \cite{Mizo}; it is a wide class of partial differential equations arising from physics and includes, fixed an integer $p \geq 2,$ all partial differential equations of the form 
\begin{equation}\label{eq}P(t,x,D_t,D_x)u:=D_tu+\sum_{j=0}^{p}\sum_{|\alpha|=j}a_{j,\alpha}D_x^\alpha u=f,\quad \ t\in[0,T],\ x\in\R^n,\quad D:=-i\partial,\end{equation}
whose principal symbol in the sense of Petrowski $\tau+\sum_{|\alpha|=p}a_{p,\alpha}\xi^\alpha=0$ admits a real root. The coefficients $a_{j,\alpha}$ may depend in general on $(t,x)$ in the linear case and also on $u$ in the nonlinear case.
The Schr\"odinger equation $i\hbar \partial_t u= -\ds\frac{\hbar^2}{2m}\triangle u+Vu,$ $t\in[0,T],\ x\in\R^n$, which models the evolution in time of the  wave function's state $u(t,x)$ of a quantic particle with mass $m$, is the most famous example of $2$-evolution equation  and it has been extensively studied in literature, both from the physical and the mathematical point of view. Going in the nonlinear realm, other relevant examples in the case $p=3, n=1$ are given by the KdV equation 
$D_t u - \frac12 \sqrt{\frac{g}{h}} \sigma D_{x}^3u + \sqrt{\frac{g}{h}} \left( \alpha+ \frac32 u \right) D_x u = 0,$ $t \in \R,\ x\in\R,$
which describes the wave motion in shallow waters, where $u(t,x)$ represents the wave elevation with respect to the water level $h$, and the so-called KdV-Burgers equation $D_t u + 2au D_x u +5ib D_x u -c D_x^3 u=0$ with $a,b,c \in \R$, which appears in the analysis of  both the flow of liquids containing gas bubbles and the propagation of waves in an elastic tube containing a viscous fluid, see \cite{JM}. For $p=5$ we mention as an example the Kawahara equation $D_t u +uD_x u-aD_x^3 u- bD_x^5 u =0$, $a,b >0$, describing magnetoacoustic waves in plasma and long water waves under ice cover, see \cite{Kawahara, Marchenko}. 
\\
All these famous models are studied in general assuming the coefficients independent of $(t,x)$; this is just a simplification, obtained by appro\-xi\-ma\-ting some physical quantities by their main value; in principle, some of the coefficients may depend on $t$ and/or $x$. A dependence on $t$ can be treated via Fourier transform, whereas to deal with equations with $x$-depending coefficients we need to make use of microlocal analysis techniques, in particular of the pseudodifferential calculus.
\\
We underline that the previous examples show that lower order terms may really appear in this class of equations, and they may be either real or complex-valued; complex-valued coefficients naturally arise also in the study of higher order (in $t$) evolution equations, see \cite{CC, beam}. From the previous examples we notice that  also the one space dimensional case $n=1$ is of great physical interest. Finally, since the first step in the study of nonlinear $p$-evolution equations is their linearization, cf. \cite{ABbis}, we shall start considering \textit{linear} $p$-evolution equations.

\medskip
In view of these considerations, in the present paper we focus on equations of the form $Pu=f$ where  
\begin{equation}\label{differential_p_evolution_operator}
	P(t,x,D_t,D_x) = D_t + a_p(t)D_x^p + \sum_{j=1}^p a_{p-j}(t,x)D_x^{p-j}, \quad (t,x)\in[0,T]\times \R,
\end{equation}
under the assumption that the leading coefficient $a_p$ is continuous on $[0,T]$, real-valued and does not vanish on $[0,T]$, whereas for $j=1,...,p-1$, we assume that $a_{p-j} \in C([0,T], \mathcal{B}^\infty(\R;\C))$, where $\mathcal{B}^\infty(\R;\C)$ denotes the space of smooth functions on $\R$ which are bounded with all their derivatives.
We are interested in the Cauchy problem associated with $P$, namely
\begin{equation}\label{cauchy_problem_gevrey}
	\left\lbrace\begin{array}{l}
		P(t,x,D_t,D_x)u(t,x)=f(t,x), \quad (t,x) \in [0,T]\times\R \\ 
		u(0,x) = g(x), \quad x\in\R
	\end{array}, \right.
\end{equation}
and the aim of the paper is to study well-posedness (i.e. existence of a unique solution and continuous dependence of the solution on the Cauchy data) of problem \eqref{cauchy_problem_gevrey} in suitable Gevrey classes.

First of all, we remark that, since $a_p$ is real-valued, $\tau=-a_p(t)\xi^p$ is the real root of the principal symbol of $P$ and the necessary condition $\mathbf{Im} \, a_p(t)\leq 0$ $\forall t\in[0,T]$ for well-posedness in $H^\infty(\R):= \cap_{s \in \R} H^s(\R)$ given in \cite{mizohata2014} is satisfied. Here, as usual, $H^s(\R)$ denotes the standard Sobolev space of all $u \in \mathscr{S}'(\R)$ such that $\langle \xi \rangle^s \hat u(\xi) \in L^2(\R),$ where $\langle \xi \rangle := \sqrt{1+\xi^2}$.
Thus, if well-posedness in $H^\infty$ fails, this is not due to the principal part but to the lower order terms.
\\
If the lower order coefficients are all real-valued, then problem \eqref{cauchy_problem_gevrey} is well-posed in all the classical functional settings ($L^2$, Sobolev spaces, Gevrey type spaces) under suitable decay assumptions on the $x$-derivatives of $a_{p-j}$ for $|x|\to \infty$, see \cite{ascanelli_chiara_zhanghirati_well_posedness_of_cauchy_problem_for_p_evo_equations}. 

The problem becomes challenging when some of the $a_{p-j}, j=1,\ldots, p-1,$ are complex-valued. For instance, in the case $p=2$ a simple computation gives 
\begin{eqnarray*}
		\frac{d}{dt} \| u(t) \|^{2}_{L^{2}} &=& 2\mathbf{Re} \, \left\langle\partial_{t} u(t), u(t)\right\rangle_{L^{2}} 
		\\ 
		&=& 2\mathbf{Re} \, \left\langle iPu(t), u(t)\right\rangle_{L^2} + 2\mathbf{Im} \, \left\langle \{ a_2(t)D_x^2 + a_{1}(t,x)D_x + a_0(t,x)\}u(t), u(t)\right\rangle_{L^2} \\
		&\leq& \|Pu(t)\|^{2}_{L^2} + C\|u\|^{2}_{L^2} + 2\mathbf{Im} \, \left\langle a_{1}(t,x) D_x u(t), u(t) \right\rangle_{L^2}. 
	\end{eqnarray*}
Since $\mathbf{Im} \, a_1(t,x) \neq 0$, the last term of the inequality above does not allow to derive an $L^2$ energy estimate in a straightforward way. It is then necessary to introduce a suitable change of variable which transforms the Cauchy problem \eqref{cauchy_problem_gevrey} into an auxiliary initial value problem for a new operator of the same form as \eqref{differential_p_evolution_operator} but with lower order terms given by positive operators whose contribution can be ignored in the application of the energy method. This approach has been introduced first in \cite{KB} and adapted and generalized in \cite{AAC3evolGelfand-Shilov, AAC3evolGevrey, CRJEECT}. The use of the above mentioned change of variable introduces in the problem relevant technical difficulties, especially in the Gevrey setting, where it requires the use of pseudodifferential operators of infinite order. Before giving more details, let us give an overview of the results existing in the literature on the problem \eqref{cauchy_problem_gevrey} with $P$ as in \eqref{differential_p_evolution_operator}. \\  \indent
The $L^2$ and $H^\infty$ theory are nowadays well understood: in the case $p=2$ necessary and sufficient conditions for well-posedness in the Sobolev setting have been given in \cite{ichinose_remarks_cauchy_problem_schrodinger_necessary_condition, I2, KB, mizohata2014}, while the general case $p\geq 2$ has been studied in \cite{ascanelli_chiara_zanghirati_necessary_condition_for_H_infty_well_posedness_of_p_evo_equations, ascanelli_chiara_zhanghirati_well_posedness_of_cauchy_problem_for_p_evo_equations }, see also \cite{ABbis} for the semilinear case. In particular, by \cite{ascanelli_chiara_zanghirati_necessary_condition_for_H_infty_well_posedness_of_p_evo_equations} we know that a necessary condition for well-posedness of \eqref{cauchy_problem_gevrey} in $H^\infty(\R)$ is the existence of
constants $M,N>0$ such that:
\begin{equation}
\label{CN2}
\sup_{x\in\R}\min_{0\leq\tau\leq t\leq T}\int_{-\varrho}^\varrho 
\mathbf{Im} \, a_{p-1}(t,x+p a_p(\tau)\theta)d\theta\leq M\log(1+\varrho)+N,\qquad 
\forall \varrho>0.
\end{equation}
\\
On the contrary in the Gevrey setting, the question of giving sufficient conditions for well posedness has been explored only in the cases $p=2$ in \cite{ACR, KB, CRJEECT} and recently in \cite{AAC3evolGevrey} for $3$-evolution equations. To describe properly the results in the Gevrey setting, we need to recall the Gevrey-Sobolev spaces.\\ 

Fixed $\theta \geq 1, m, \rho \in \R,$  we set
$$
H^{m}_{\rho; \theta} (\R) = \{ u \in \mathscr{S}'(\R) :  \langle D \rangle^{m} 
e^{\rho \langle D \rangle^{\frac{1}{\theta}}} u \in L^{2}(\R) \},
$$
where $\langle D \rangle^m$ and $e^{\rho \langle D \rangle^{\frac{1}{\theta}}}$ are the Fourier multipliers with symbols $\langle \xi \rangle^m$ and $e^{\rho \langle \xi \rangle^{\frac{1}{\theta}}}$ respectively.
These spaces are Hilbert spaces with the following inner product
$$\langle u, v \rangle_{H^{m}_{\rho;\theta}} = \, \langle\langle D \rangle^{m} e^{\rho\langle D \rangle^{\frac{1}{\theta}}}u, \langle D \rangle^{m} e^{\rho\langle D \rangle^{\frac{1}{\theta}}}v \rangle_{L^{2}}, \quad u, v \in H^{m}_{\rho;\theta}(\R).
$$
We set
$$ \mathcal{H}^\infty_{\theta} (\R): = \bigcup_{\rho >0}H^m_{\rho; \theta}(\R). $$
The following inclusions hold:
$$ G_0^\theta(\R) \subset  \mathcal{H}^\infty_{\theta} (\R) \subset G^\theta(\R),$$
where $G^\theta(\R)$ 
denotes the space of all smooth functions $f$ on $\R$ such that 
\begin{equation}
	\label{gevestimate}
	\sup_{\alpha \in \N_0} \sup_{x \in \R} h^{-|\alpha|} \alpha!^{-\theta} |\partial^\alpha f(x)| < +\infty
\end{equation} 
for some $h >0$ and $G_0^\theta(\R)$ is the space of all compactly supported functions contained in $G^\theta(\R)$.

 Coming now to necessary conditions in the Gevrey setting,
the first three authors in \cite{AACpevolGevreynec} have proved for the operator \eqref{differential_p_evolution_operator} the following:

\medskip
\noindent {\it {\bf Theorem.} If the Cauchy problem \eqref{cauchy_problem_gevrey} is well-posed in $\mathcal{H}^{\infty}_{\theta}(\R)$, $\theta > 1$, and:
	\begin{itemize}
		\item [(i)] there exist $R, A > 0$ and $\sigma_{p-j} \in [0,1]$, $j = 1, \ldots, p-1,$ such that 
		$$
		\mathbf{Im} \, a_{p-j}(t,x) \geq A \langle x \rangle^{-\sigma_{p-j}}, \quad x > R \, (\textrm{or} \, \, \, x <-R), \, t \in [0,T], \, j=1,\ldots,p-1; $$
		\item [(ii)] there exists $C > 0$ such that for every $\beta \in \N$:
		$$
		|\partial^{\beta}_{x}a_{p-j}(t,x)| \leq C^{\beta + 1} \beta! \langle x \rangle^{-\beta}, \quad x \in \R, \, t \in [0,T],\quad j = 1, \ldots, p,
		$$
	\end{itemize}
 then 
	\begin{equation}\label{equation_thesis_main_theorem}
		\Xi := \max_{j=1,\ldots,p-1} \{(p-1)(1-\sigma_{p-j}) - j + 1\} \leq \frac{1}{\theta}.
	\end{equation}
}
The result above gives some necessary conditions on the decay rates of the coefficients for the Gevrey well-posedness. Since $\Xi \geq 0$ and $\frac{1}{\theta} < 1$, the following considerations are a consequence of \eqref{equation_thesis_main_theorem}:
\begin{itemize}
\item[-]  if $\sigma_{p-j} \leq \frac{p-1-j}{p-1}$ for some $j=1,\ldots, p-1,$ the Cauchy problem is not well-posed in $\mathcal{H}^{\infty}_{\theta}(\R)$ ;
\item[-] if $\sigma_{p-j} \in \left( \frac{p-1-j}{p-1},  \frac{p-j}{p-1}\right)$ for some $j=1,\ldots, p-1,$ then the power $\sigma_{p-j}$ imposes the restriction 
		$
		(p-1)(1-\sigma_{p-j}) - j + 1 \leq \frac{1}{\theta}
		$
		for the indices $\theta$ where $\mathcal{H}^{\infty}_{\theta}(\R)$ well-posedness can be found;
\item[-] if 
 $\sigma_{p-j} \geq \frac{p-j}{p-1}$ for some $j=1,\ldots, p-1,$ then the power $\sigma_{p-j}$ has no effect on the $\mathcal{H}^{\infty}_{\theta}(\R)$ well-posedness.
\end{itemize}
With this result in mind we can better understand the sufficient conditions given in \cite{KB,AAC3evolGevrey} for the cases $p=2,3.$ In short, in the case $p=2$ in \cite{KB} the coefficient $a_1$ of the lower term is assumed to be Gevrey regular of order $\theta _0 >1$ and decaying at infinity like $|x|^{-\sigma}$ for some $\sigma \in (0,1)$ and well-posedness in $\mathcal{H}_\theta^\infty(\R)$ is achieved for $\theta_0 \leq \theta < (1-\sigma)^{-1}$. In the case $p=3$, in \cite{AAC3evolGevrey}, assuming $\sigma \in (1/2,1), \theta_0 < \frac{1}{2(1-\sigma)}$ and 
\begin{itemize}
	\item [(i)] $a_3 \in C([0,T]; \R)$ and $\exists C_{a_3} > 0$ such that $|a_3(t)| \geq C_{a_3}\ \forall t \in [0,T]$,
	\item [(ii)] $a_{p-j} \in C([0,T]; G^{\theta_0}(\R))$, $\theta_0 > 1$, for $j = 1,2,3$,
	\item [(iii)] $\exists C_{a_2} > 0$ such that  $$|\partial^{\beta}_{x} a_2(t,x)| \leq C^{\beta+1}_{a_2} \beta!^{\theta_0} \langle x \rangle^{-\sigma}, \quad \forall t \in [0,T], \, x  \in \R, \beta \in \N_{0},$$
	\item [(iv)] $\exists C_{a_1}$ such that $|\mathbf{Im} \, a_1(t,x)| \leq C_{a_1}  \langle x \rangle^{-\frac{\sigma}{2}}$ for every $t \in [0,T], \, x \in \R$,
\end{itemize}
well-posedness in $\mathcal{H}^\infty_\theta(\R)$ for $\theta \in \left[ \theta_0, \frac1{2(1-\sigma)}\right)$ holds.

\medskip
In the present paper, taking into account the necessary conditions proved in \cite{AACpevolGevreynec}, we generalize the results in \cite{AAC3evolGevrey} to the $p$-evolution operator \eqref{differential_p_evolution_operator}, where $ p \geq 2$. We underline that the generalization is far from being straightforward: it needs a careful use of sharp G{\aa}rding inequality, several nontrivial technical steps and iterative procedures.
\\ \indent 
 The main result of the paper reads as follows.

\begin{theorem}\label{maintheorem1}
	Let $\theta_0>1$ and $\sigma \in \left(\frac{p-2}{p-1},1\right)$ such that $\theta_0 < \frac{1}{(p-1)(1-\sigma)}$. Let $P$ be an operator of the type \eqref{differential_p_evolution_operator} whose coefficients satisfy the following assumptions:
	\begin{itemize}	\item[\textup{(i)}] $a_p \in C([0,T];\R)$ and there exists $C_{a_p}>0$ such that $|a_p(t)| \geq C_{a_p}$, for all $t \in [0,T]$.
		\item[\textup{(ii)}] $|\partialx^\beta a_{p-j}(t,x)| \leq C_{a_{p-j}}^{\beta+1} \beta!^{\theta_0} \japx^{-\frac{p-j}{p-1}\sigma-\beta}$, for some $C_{a_{p-j}}>0$, $j=1,...,p-1$ and for all $\beta \in \N_0$, $(t,x) \in [0,T] \times \R$.
	\end{itemize}
	If $\theta>1$ is such that $\theta_0 \leq \theta < \frac{1}{(p-1)(1-\sigma)}$, the data $f \in C\left([0,T];H_{\rho;\theta}^m(\R)\right)$ and $g \in H_{\rho;\theta}^m(\R)$, with $m,\rho\in\R$ and $\rho>0$, then the Cauchy problem \eqref{cauchy_problem_gevrey} admits a unique solution $u \in C\left([0,T];H_{\tilde \rho;\theta}^m(\R)\right)$ for some $\tilde \rho \in (0,\rho)$, and the solution satisfies the energy estimate
	\begin{equation}\label{energy_estimate_cauchy_gevrey}
		\| u(t) \|_{H_{\tilde \rho;\theta}^m}^2 \leq C \left( \| g \|^2_{H_{\rho;\theta}^m} + \int_0^t \| f(\tau) \|_{H_{\rho;\theta}^m}^2 d\tau \right),
	\end{equation}
	for all $t\in[0,T]$ and for some constant $C>0$. In particular, for $\theta\in\left[\theta_0,\frac{1}{(p-1)(1-\sigma)}\right)$ the Cauchy problem \eqref{cauchy_problem_gevrey} is well-posed in $\mathcal{H}_\theta^\infty(\R)$.
\end{theorem}

\begin{remark}
	Notice that the assumption (ii) of Theorem \ref{maintheorem1} is not sufficient to guarantee well-posedness in $L^2(\R)$, nor in $H^\infty(\R)$ for the Cauchy problem \eqref{cauchy_problem_gevrey}, since the necessary condition \eqref{CN2} is not satisfied in general under this assumption.
\end{remark}

\begin{remark}
Comparing the assumptions of Theorem 1 with the known literature for 3-evolution equations, see e.g. \cite[Theorem 1.2]{AAC3evolGevrey}, we notice that our decay at infinity condition (ii) is given on the whole coefficients of the lower order terms (without distinguishing the behavior of the real and of the imaginary parts) and prescribes for the derivatives of the coefficients a decay which increases with the order of the derivatives. This extra decay  is used in Subsection 4.1 to absorb a growth in $x$ of some terms appearing for instance in \eqref{case_1_equation_4} and \eqref{case_1_equation_5}. Without assuming the decay of the derivatives, the estimate of these terms is possible but it is more involved. Moreover, under the condition (ii), the symbols $a_{p-j}(t,x)\xi^{p-j}$ are symbols in $\textrm{\textbf{SG}}$ classes (cf. page 6 for the definition), and this allows to apply \cite[Theorem 6]{ACsharpGarding} in Section \ref{estimatesfrombelow}. A version of this theorem for Gevrey regular H\"ormander symbols with a precise estimate of the regularity of the remainders is still missing in the literature as far as we know although it is a somewhat expected result. In conclusion, we believe that we could have distinguished conditions on real and imaginary parts of the coefficients as in \cite[Theorem 1.1]{ascanelli_chiara_zhanghirati_well_posedness_of_cauchy_problem_for_p_evo_equations} or \cite[Theorem 1.2]{AAC3evolGevrey}, and we could have assumed (ii) only for a finite number of derivatives or avoided the extra decay for the $x$-derivatives, but we preferred to skip these refinements in order to not add further technicality to the proof and to work in the frame of the standard $\textrm{\textbf{SG}}$ calculus. \end{remark}


\begin{remark} In this paper the results concern $p$-evolution equations in one space dimension. The main difficulty in the extension to the higher space dimensional case consists in the definition of an appropriate change of variable which becomes much more involved in this case. At this moment the only results concerning evolution operators with coefficients defined on $[0,T] \times \R^n, n > 1,$ are limited to the case $p=2$, cf. \cite{Arias_GS, ACJMPA, KB, CRJEECT}. 
\end{remark}

The paper is organized as follows. In section \ref{sec:preliminaries} we introduce some classes of pseudodifferential operators that we will use to prove our main result. As we are dealing with Gevrey classes, some of these operators will be of infinite order, that is, characterized by an exponential growth of the symbol.
In Section \ref{sec:changeofvariable} we introduce a suitable change of variable which allows to reduce the Cauchy problem \eqref{cauchy_problem_gevrey} to an auxiliary initial value problem via a conjugation. Section \ref{conjugation} is devoted to performing this conjugation. In Section \ref{estimatesfrombelow} we prove that the new problem obtained is well-posed in standard Sobolev spaces. 
In the last section, going back with the inverse change of variable, we  prove Theorem \ref{maintheorem1} 
and close the paper with a remark concerning the sufficiency counterpart of Theorem $1$ of \cite{AACpevolGevreynec} (see Remark \ref{bramble_scramble}).

\section{Pseudodifferential operators} \label{sec:preliminaries}
In this section we introduce the pseudodifferential operators that we will employ in the study of the Cauchy problem \eqref{cauchy_problem_gevrey}. Although we shall apply these tools to a problem in one space dimension in the present paper, the next definitions and results are valid in arbitrary dimension so we prefer to give them in this general case in view of future applications. More details can be found in \cite{CappielloRodino, Rodino_linear_partial_differential_operators_in_gevrey_spaces}.
\\

Fixed $ \mu > 1$, $A>0$ and $m,m_1,m_2 \in \R$ we will consider the following Banach spaces: 
$$p(x,\xi) \in {\bf S}^{m}_{\mu}(\R^{2n};A) \iff 
\sup_{\overset{\alpha, \beta \in \N_{0}^{n}}{x,\xi \in \R^{n}}} |\partial_\xi^\alpha \partial_x^\beta p(x,\xi)| A^{-|\alpha+\beta|} (\alpha! \beta!)^{-\mu} \langle \xi \rangle^{-m+|\alpha|} < +\infty,
$$
$$
p(x,\xi) \in \tilde{{\bf S}}^{m}_{\mu}(\R^{2n};A) \iff |p|_{A} := \sup_{\overset{\alpha, \beta \in \N_{0}^{n}}{x,\xi \in \R^{n}}} |\partial_\xi^\alpha \partial_x^\beta p(x,\xi)| A^{-|\alpha+\beta|} (\alpha! \beta!)^{-\mu} \langle \xi \rangle^{-m} < +\infty,
$$
$$ p \in \textrm{\textbf{SG}}^{m_1,m_2}_{\mu}(\R^{2n}; A) \Leftrightarrow 
\sup_{\stackrel{\alpha, \beta \in \N^{n}_{0}}{x,\xi \in \R^{n}} }| \partial_\xi^\alpha \partial_x^\beta p(x,\xi) |A^{-|\alpha+\beta|} (\alpha!\beta!)^{-\mu} \langle \xi \rangle^{-m_1+|\alpha|} \langle x \rangle^{-m_2+|\beta|}  <+\infty.
$$ Since the classes above become larger when $A$ increases, we can set  
\begin{equation}\label{indsymb1}{\bf S}^m_{\mu}(\R^{2n}):= \bigcup_{A>0}{\bf S}^{m}_{\mu}(\R^{2n};A), \qquad \tilde{{\bf S}}^m_{\mu}(\R^{2n}):= \bigcup_{A>0}\tilde{{\bf S}}^{m}_{\mu}(\R^{2n};A),\end{equation}
\begin{equation}\label{indsymb2} \textbf{\textrm{SG}}^{m_1,m_2}_{\mu}(\R^{2n}):= \bigcup_{A>0}\textrm{\bf SG}^{m_1,m_2}_{\mu}(\R^{2n};A) \end{equation} endowed with the inductive limit topology.

Notice that the classes $\textrm{\bf SG}^{m_1, m_2}_\mu(\R^{2n})$ are characterized by the fact of having two orders, related to the behavior with respect to the variables $\xi$ and  $x$ respectively. We shall frequently refer to these orders as order w.r.t $\xi$ and order w.r.t $x$ in the sequel. \\
 
Given a symbol $p \in {\bf \tilde S}^{m}_{\mu}(\R^{2n})$ we denote by $p(x,D)$ or by $\mathbf{op}(p)$ the pseudodifferential operator 
\begin{equation} \label{pseudop}
	p(x,D) u (x) = \int e^{i\xi x} p(x,\xi) \widehat{u}(\xi) \dslash\xi, \quad u \in \mathscr{S}(\R^{n}),
\end{equation}
where $\dslash\xi = (2\pi)^{-n}d\xi$.
 From the classical theory of pseudodifferential operators, $p(xD)$ extends to a linear and continuous operator from $H^{m'}(\R^n)$ to $H^{m'-m}(\R^n)$. For our purposes, we need to recall the action of $p(x,D)$ on the Gevrey-Sobolev spaces defined in the Introduction. We have the following result, cf. \cite[Proposition 6.2]{KN}.

\begin{proposition}\label{prop_continuity_finite_order_gevrey_sobolev}
	Let $p \in \tilde{{\bf S}}^{m}_{\mu}(\R^{2n};A)$ for some $A>0$. Then:
	\begin{itemize}
		\item[i)] $p(x,D)$  maps continuously $H^{m+m'}_{\rho;\theta}(\R^n)$ into $H^{m'}_{\rho;\theta}(\R^n)$ for every $m' \in \R, \rho \in \R$ and $\theta >\mu$.
		\item[ii)] There exists $\delta >0$ such that if $|\rho|< \delta A^{-1/\theta}$, $p(x,D)$ maps continuously  $H^{m+m'}_{\rho;\mu}(\R^n)$ into $H^{m'}_{\rho;\mu}(\R^n).$
		\end{itemize}
\end{proposition}

By \cite[Proposition 6.4]{KN}, given $p \in {\bf S}^{m}_{\mu}(\R^{2n};A)$ and $q \in {\bf S}^{m'}_{\mu}(\R^{2n};A)$, the operator $p(x,D)q(x,D)$ is a pseudodifferential operator with symbol $s$ given for every $N \geq 1$ by
$$s(x,\xi)= \sum_{|\alpha| <N}(\alpha!)^{-1}\partial_\xi^\alpha p(x,\xi) D_x^\alpha q(x,\xi) +r_N(x,\xi),$$
where $r_N \in {\bf S}_\mu^{m+m'-N}(\R^{2n}).$

In the sequel we shall also consider symbols of the type $e^{\Lambda}$, where $\Lambda \in \textbf{S}^{1/\kappa}(\R^2)$ is real-valued and $\kappa > 1$. Operators coming from this type of symbols are very helpful in the analysis of evolution operators in Gevrey type spaces, see \cite{AAC3evolGelfand-Shilov, AAC3evolGevrey, scncpp2, ACJMPA, CRJEECT, KB, KN, Zanghirati}. 
It is easy to verify that $e^{\pm \Lambda}$ satisfies an estimate of the form
\begin{equation}\label{expest}
	|\partial_\xi^\alpha \partial_x^\beta e^{\pm \Lambda(x,\xi)}| \leq  A_1^{|\alpha+\beta|}\langle \xi \rangle^{-|\alpha|}(\alpha!\beta!)^\mu e^{2\rho_0\langle \xi \rangle^{\frac1{\kappa}}}
\end{equation}
for some positive constant $A_1$ independent of $\alpha, \beta$, where 
$$
\rho_0:= \sup_{(\alpha, \beta) \in \N_{0}^{2n}} \sup\limits_{(x,\xi) \in \R^{2n}} A^{-|\alpha+\beta|} (\alpha!\beta!)^{-\mu} \langle \xi \rangle^{-1/\kappa + |\alpha|}|\partial^{\alpha}_{\xi}\partial^{\beta}_{x} \Lambda(x,\xi)|,
$$ 
see \cite[Lemma 6.2]{KN}. The estimate \eqref{expest} guarantees that the related pseudodifferential operator
$$e^{\pm \Lambda}(x,D) u(x) = \int_{\R^n} e^{i\xi x \pm\Lambda(x,\xi)}\hat{u}(\xi)\, \dslash \xi$$ 
is well defined and continuous as an operator from $G_0^\theta(\R^n)$ to $G^\theta(\R^n)$ for every $\theta \in (\mu, \kappa)$, see \cite[Theorem 3.2.3]{Rodino_linear_partial_differential_operators_in_gevrey_spaces} or \cite[Theorem 2.4]{Zanghirati}. 
We shall also consider the so-called reverse operator of $e^{\pm \Lambda}(x,D)$, denoted by $^{R} (e^{\pm \Lambda}(x,D))$. This operator, introduced in \cite[Proposition 2.13]{KW} as the transposed of $e^{\pm \Lambda}(x,-D)$,  is defined as an oscillatory integral by
$$^{R}(e^{\pm \Lambda}(x,D))u(x) = Os - \iint e^{i\xi (x-y) \pm \Lambda(y,\xi)}u(y)\, dy \dslash \xi.
$$
The following continuity result holds for the operators $e^\Lambda(x,D)$ and $^R (e^\Lambda(x,D))$, cf. \cite[Proposition 2.5]{AAC3evolGevrey}.

\begin{proposition}\label{contgev}
	Let $\Lambda \in \tilde{{\bf S}}_\mu^{1/\kappa}(\R^{2n};A)$ for some $A>0$ and $\kappa, \mu \in \R$ such that $1<\mu < \kappa$ and let  $\rho,m\in\R$ and $\theta  \in (\mu,\kappa)$. Then the operators $e^{\Lambda}(x,D)$ and $^R (e^{\Lambda}(x,D))$ map continuously $H^{m}_{\rho; \theta} (\R^n)$ into $H^{m}_{\rho-\delta; \theta}(\R^n)$ for every $\delta >0$.	
\end{proposition}
%
%
%
%
%

In the next result we shall need to work with the weight function $\langle\xi\rangle_h= (h^2+|\xi|^2)^{1/2}$
where $h \geq 1$. We point out that we can replace $\langle \xi \rangle$ by $\langle \xi \rangle_{h}$ in all previous definitions and statements, and this replacement does not change the dependence of the constants, that is, all the previous constants are independent of $h$. 

We have the following conjugation theorem which has been proved in \cite[Theorem 2.8]{AAC3evolGevrey}.
\begin{theorem}\label{theorem_conjugation}
	Let $p \in {\bf S}^m_{\kappa}(\R^{2n};A)$ for some $A>0$
	and let $\Lambda$ satisfy for some $1 < \mu < \kappa$:
	\begin{equation}\label{firstasslambda}
		|\partial_\xi^\alpha \Lambda(x,\xi)| \leq \rho_0 A^{|\alpha|} \alpha!^{\mu} \langle \xi \rangle^{\frac{1}{\kappa}-|\alpha|}_{h}
	\end{equation}
	and 
	\begin{equation}\label{equation_stronger_hypothesis_on_Lambda}
		|\partial_\xi^\alpha \partial_x^\beta \Lambda(x,\xi)| \leq \rho_0 A^{|\alpha+\beta|}\alpha!^{\mu}\beta!^{\mu} \langle \xi \rangle^{-|\alpha|}_{h}, 
	\end{equation}
	whenever $|\beta| \geq 1$. 
	Then there exist $\tilde{\delta} > 0$ and $h_0 = h_0(A) \geq 1$ such that if $\rho_0 \leq \tilde{\delta} A^{-\frac{1}{\kappa}}$ and $h \geq h_0$, then 
	\begin{multline}\label{asymptotic_expansion}
		e^\Lambda(x,D) p(x,D) ^{R}(e^{-\Lambda}(x,D)) = p(x,D)+ \mathbf{op} \left( \sum_{1 \leq |\alpha+\beta| < N} \frac{1}{\alpha!\beta!} \partial^{\alpha}_{\xi} \{\partial^{\beta}_{\xi} e^{\Lambda(x,\xi)} D^{\beta}_{x}p(x,\xi) D^{\alpha}_{x}e^{-\Lambda(x,\xi)} \} \right)  \\ 
		+ r_{N}(x,D) + r_{\infty}(x,D) ,
	\end{multline}
	where 
	\begin{equation*}
		|\partial^{\alpha}_{\xi}\partial^{\beta}_{x}r_N(x,\xi)| \leq C_{\rho_0,A,\kappa} (C_{\kappa}A)^{|\alpha+\beta|+2N}\alpha!^{\kappa}\beta!^{\kappa}N!^{2\kappa-1} \langle \xi \rangle_h^{m-(1-\frac{1}{\kappa})N - |\alpha|},
	\end{equation*}
	\begin{equation} \label{regularizingestimate}
		|\partial^{\alpha}_{\xi}\partial^{\beta}_{x}r_{\infty}(x,\xi)| \leq C_{\rho_0,A,\kappa} (C_{\kappa}A)^{|\alpha+\beta|+2N}\alpha!^{\kappa}\beta!^{\kappa}N!^{2\kappa-1} e^{-c_\kappa A^{-\frac{1}{\kappa}} \langle \xi \rangle_h^{\frac{1}{\kappa}}}.
	\end{equation}
	\end{theorem}

\section{Construction of a change of variable}\label{sec:changeofvariable}

As mentioned in the Introduction, in order to prove Theorem \ref{maintheorem1}, we are going to perform
a change of variable which turns the Cauchy problem \eqref{cauchy_problem_gevrey} into an equivalent Cauchy problem which is well-posed in Sobolev spaces. The change of variable will have the form $$v:=Qu,$$ where $Q$ is a suitable invertible pseudodifferential operator of infinite order. By the equivalence $u=Q^{-1}v$ we will recover the solution $u$ to the original Cauchy problem from the solution $v$ to the auxiliary problem. The infinite order of the operator $Q^{-1}$ will determine the space where the solution $u$ belongs to.
In what follows we are going to explain the structure of the operator $Q$: it will be given by the composition of two operators, and we describe for each term of the composition its role in the transformation.

 The operator $Q$ will have the following structure:
\begin{equation}\label{operator_of_conjugation_gevrey}
	Q(t,x,D) = Q_{\Lambda,K,\rho^\prime}(t,x,D) := e^{\Lambda_{K,\rho^\prime}}(t,D) \circ e^\Lambda(x,D),
\end{equation}
 where
$$
\Lambda(x,\xi) = \sum_{k=1}^{p-1}\lambda_{p-k}(x,\xi) \in \mathbf{SG}_\mu^{0,1-\sigma}(\R^2) \cap \mathbf{S}_\mu^{(p-1)(1-\sigma)}(\R^2)
$$
for some $\mu>1$, and 
$$
\Lambda_{K,\rho^\prime}(t,\xi) = K(T-t)\japxih^{(p-1)(1-\sigma)} + \rho^\prime\japxih^{1/\theta}
$$
with $0<\rho^\prime<\rho$, with $\rho$ coming from the Cauchy data, $K>0$ and $h >>1$ large to be both chosen later on.
\\

Now, let us describe each part of the operator $Q_{\Lambda,K,\rho^\prime}(t,x,D)$.

\begin{itemize}
	\item The role of each factor $e^{\lambda_{p-j}}$ of the symbol $e^\Lambda=e^{\lambda_{p-1}}\cdot \ldots \cdot e^{\lambda_{p-1}}$ in the conjugation with $e^\Lambda(x,D)$ is to turn $\mathbf{Im} \, a_{p-j}(t,x,D), j=1,\ldots,p,$ into the sum of a positive operator plus a term of lower order, without changing the parts of order $p,p-1,...,p-j+1$ of the operator. Summing up, the conjugation with $e^\Lambda(x,D)$ turns the operator $P(t,x,D_t,D_x)$ into a sum of positive operators plus a remainder of order $(p-1)(1-\sigma)$.
	
	\item The operator $e^{K(T-t)\langle D \rangle_h^{(p-1)(1-\sigma)}}$ does not change terms of order $1, ..., p$, but it corrects the error of order $(p-1)(1-\sigma)$ coming from the previous transformation by changing it into the sum of a positive operator plus a remainder of order zero. This is obtained by choosing $K$ sufficiently large.
	
	\item The term $e^{\rho^\prime \langle D \rangle_h^{1/\theta}}$ is the leading term of the whole transformation $Q_{\Lambda,K,\rho^\prime}(t,x,D)$, since we are assuming $(p-1)(1-\sigma)<1/\theta$: it changes the setting of the Cauchy problem from Gevrey-Sobolev type spaces to the standard Sobolev spaces. We remark that, since $\rho'>0$, $e^{\rho^\prime \japxih^{1/\theta}}$ is the leading part of $Q_{\Lambda,K,\rho^\prime}(t,x,\xi)$, so the inverse operator $(Q_{\Lambda,K,\rho^\prime}(t,x,D))^{-1}$ has regularizing properties with respect to the spaces $H_{\rho;\theta}^m(\R)$, i.e. it maps $H^m(\R)$ into a Gevrey-Sobolev space $H^m_{\rho'-\delta;\theta}(\R)$ for every positive $\delta$.
\end{itemize}
In the following subsections we are going to give more details about this change of variable. Setting
$$
P_{\Lambda,K,\rho^\prime}(t,x,D_t,D_x) := Q_{\Lambda,K,\rho^\prime}(t,x,D_x) \circ P(t,x,D_t,D_x) \circ (Q_{\Lambda,K,\rho^\prime}(t,x,D_x))^{-1},
$$
we note that the Cauchy problem \eqref{cauchy_problem_gevrey} is equivalent to the auxiliary Cauchy problem
\begin{equation}\label{auxiliary_cauchy_problem_gevrey}
	\left\lbrace\begin{array}{ll}
		P_{\Lambda,K,\rho^\prime}(t,x,D_t,D_x)v(t,x)=Q_{\Lambda,K,\rho^\prime}(t,x,D_x)f(t,x), & (t,x)\in[0,T]\times\R \\ 
		v(0,x)=Q_{\Lambda,K,\rho^\prime}(0,x,D_x)g(x), & x\in\R
	\end{array}, \right.
\end{equation}
in the sense that if $u$ solves \eqref{cauchy_problem_gevrey}, then $v=Q_{\Lambda,K,\rho^\prime}(t,x,D_x)u$ solves \eqref{auxiliary_cauchy_problem_gevrey} and, if $v$ solves \eqref{auxiliary_cauchy_problem_gevrey}, then $u=(Q_{\Lambda,K,\rho^\prime}(t,x,D_x))^{-1}v$ solves \eqref{cauchy_problem_gevrey}.

\subsection{The functions $\lambda_{p-k}(x,\xi)$, $k=1,...,p-1$}

For each $k=1,...,p-1$, let $M_{p-k}>0$ to be chosen later on and define
\begin{equation}\label{function_lambda_(p-k)}
	\lambda_{p-k}(x,\xi):=M_{p-k}\omega\left(\frac{\xi}{h}\right)\japxih^{1-k}\int_0^x\langle y\rangle^{-\frac{p-k}{p-1}\sigma}\psi\left(\frac{\langle y\rangle}{\japxih^{p-1}}\right)dy,
\end{equation}
where $\omega$ and $\psi$ are $C^\infty$ functions such that
\begin{equation}\label{z}
\omega(\xi) = \left\lbrace \begin{array}{ll}
	0, & |\xi|\leq1 \\ 
	-\text{sgn}(a_p(t)), & |\xi|>R_{a_p}
\end{array}  \right., \qquad \psi(y) = \left\lbrace \begin{array}{ll}
	1, & |y|\leq\frac{1}{2} \\ 
	0, & |y|\geq1
\end{array}  \right.
\end{equation}
and 
$$
|\partialxi^\alpha\omega(\xi)|\leq C_\omega^{\alpha+1}\alpha!^\mu\qquad\text{and}\qquad |\partial_y^\beta\psi(y)|\leq C_\psi^{\beta+1}\beta!^\mu.
$$
for some fixed $\mu>1$ (that we can take very close to $1$).
Notice that, by assumption (i) in Theorem \ref{maintheorem1}, $\omega$ is constant for $|\xi|\geq R_{a_p}$, hence, if $\alpha\ne0$ we have that $\omega^{(\alpha)}\left(\frac{\xi}{h}\right)$ is supported in $|\xi|/h \leq R_{a_p}$, which gives us
$$
h^{-\alpha}\leq \japxih^{-\alpha}\langle R_{a_p}\rangle^\alpha.
$$


The following result collects several alternative estimates satisfied by the function $\lambda_{p-k}$. In the proof of our results we shall use at each step the most convenient of them.

\begin{lemma}\label{estimatesforlambda}
	For each $k=1,...,p-1$, the following statements hold:
	\begin{itemize}
		\item[\textup{(i)}] $|\lambda_{p-k}(x,\xi)|\leq\frac{M_{p-k}}{1-\frac{p-k}{p-1}\sigma}\japxih^{(p-k)(1-\sigma)}.$
		\item[\textup{(ii)}] $|\partialxi^\alpha\lambda_{p-k}(x,\xi)|\leq C^{\alpha+1}\alpha!^\mu\japxih^{(p-k)(1-\sigma)-\alpha}$, for all $\alpha\geq1$.
		\item[\textup{(iii)}] $|\partialxi^\alpha\lambda_{p-k}(x,\xi)|\leq C^{\alpha+1}\alpha!^\mu\japxih^{1-k-\alpha}\japx^{1-\frac{p-k}{p-1}\sigma}$, for all $\alpha\geq0$.
		\item[\textup{(iv)}] $|\partialxi^\alpha\lambda_{p-k}(x,\xi)|\leq C^{\alpha+1}\alpha!^\mu\japxih^{-\alpha}\japx^{\frac{p-k}{p-1}(1-\sigma)}$, for all $\alpha\geq0$.
		\item[\textup{(v)}] $|\doublepartial\lambda_{p-k}(x,\xi)|\leq C^{\alpha+\beta+1}(\alpha!\beta!)^\mu\japxih^{1-k-\alpha}\japx^{-\frac{p-k}{p-1}\sigma-(\beta-1)}$, for all $\alpha\geq0$ and $\beta\geq1$.
	\end{itemize}
\end{lemma}

\begin{proof}
	Let us denote by $\chi_\xi(x)$ the characteristic function of the set $\{x\in\R^n;\japx\leq\japxih^{p-1}\}$. Then, it follows that
	\begin{eqnarray}
		|\lambda_{p-k}(x,\xi)| & = & M_{p-k}\left|\omega\left(\frac{\xi}{h}\right)\right|\japxih^{-k+1}\left|\int_0^x \langle y\rangle^{-\frac{p-k}{p-1}\sigma}\psi\left(\frac{\langle y\rangle}{\japxih^{p-1}}\right)dy\right| \nonumber \\
		& \leq & M_{p-k}\japxih^{-k+1}\int_0^{|x|}\langle y\rangle^{-\frac{p-k}{p-1}\sigma}\chi_\xi(y)dy \nonumber \\
		& \leq & M_{p-k}\japxih^{-k+1}\int_0^{\min\{|x|,\japxih^{p-1}\}}\langle y\rangle^{-\frac{p-k}{p-1}\sigma}dy \nonumber \\
		& \leq & M_{p-k}\japxih^{-k+1}\int_0^{\min\{\japx,\japxih^{p-1}\}}y^{-\frac{p-k}{p-1}\sigma}dy \nonumber \\
		& = & \frac{M_{p-k}}{1-\frac{p-k}{p-1}\sigma}\japxih^{-k+1}\left[y^{1-\frac{p-k}{p-1}\sigma}\right]_0^{\min\{\japx,\japxih^{p-1}\}} \nonumber \\
		& = & \frac{M_{p-k}}{1-\frac{p-k}{p-1}\sigma}\japxih^{-k+1}(\min\{\japx,\japxih^{p-1}\})^{1-\frac{p-k}{p-1}\sigma} \nonumber \\
		& \leq & \frac{M_{p-k}}{1-\frac{p-k}{p-1}\sigma}\japxih^{-k+1}(\japxih^{p-1})^{1-\frac{p-k}{p-1}\sigma} \nonumber \\
		& = & \frac{M_{p-k}}{1-\frac{p-k}{p-1}\sigma}\japxih^{-k+1}\japxih^{p-1-(p-k)\sigma} \nonumber \\
		& = & \frac{M_{p-k}}{1-\frac{p-k}{p-1}\sigma}\japxih^{(p-k)(1-\sigma)} \nonumber
	\end{eqnarray}
	which gives us (i). From the previous computation we also obtain
	$$
	|\lambda_{p-k}(x,\xi)|\leq\frac{M_{p-k}}{1-\frac{p-k}{p-1}\sigma}\japxih^{1-k}\japx^{1-\frac{p-k}{p-1}\sigma}.
	$$

	Now, we want to obtain the estimate (ii) for $|\partialxi^\alpha\lambda_{p-k}(x,\xi)|$. Note that, by Leibniz rule:
	\begin{eqnarray}\label{6}
		\partialxi^\alpha\lambda_{p-k}(x,\xi) & = & \partialxi^\alpha\left\lbrace M_{p-k}\omega\left(\frac{\xi}{h}\right)\japxih^{1-k}\int_0^x\langle y\rangle^{-\frac{p-k}{p-1}\sigma}\psi\left(\frac{\langle y\rangle}{\japxih^{p-1}}\right)dy\right\rbrace \nonumber \\
		& = & M_{p-k}\sum_{\alpha_1+\alpha_2+\alpha_3=\alpha}\frac{\alpha!}{\alpha_1!\alpha_2!\alpha_3!}\left[\partialxi^{\alpha_1}\omega\left(\frac{\xi}{h}\right)\right]\left[\partialxi^{\alpha_2}\japxih^{-1+k}\right] \nonumber \\
		& \times & \int_0^x\langle y\rangle^{-\frac{p-k}{p-1}\sigma}\left[\partialxi^{\alpha_3} \psi\left(\frac{\langle y\rangle}{\japxih^{p-1}}\right)\right]dy
	\end{eqnarray}
	Now we need to find a way to deal with the derivatives which appear in \eqref{6}. The first one is very simple; the second one can be estimated by $\partialxi\japxih^m\leq C_m^{\alpha}\alpha!\japxih^{m-\alpha}$, $\xi\in\R$, $\alpha\in\N$, where $C_m$ is a positive constant independent of $h$; the third derivative needs to be computed by using Fa\`a di Bruno's formula in the following way
	$$
	\partialxi^{\alpha_3}\psi\left(\frac{\langle y\rangle}{\japxih^{p-1}}\right)=\sum_{j=1}^{\alpha_3}\frac{1}{j!}\psi^{(j)}\left(\frac{\langle y\rangle}{\japxih^{p-1}}\right)\sum_{\gamma_1+\cdots+\gamma_j=\alpha_3}\frac{\alpha_3!}{\gamma_1!\cdots\gamma_j!}\prod_{\ell=1}^j\partialxi^{\gamma_\ell}\left(\frac{\langle y\rangle}{\japxih^{p-1}}\right).
	$$
	Hence, we can estimate in \eqref{6} as follows:
	\begin{eqnarray*}
		|\partialxi^\alpha\lambda_{p-k}(x,\xi)| & \leq & M_{p-k}\sum_{\alpha+\alpha_2+\alpha_3=\alpha}\frac{\alpha!}{\alpha_1!\alpha_2!\alpha_3!}\left|\omega^{(\alpha_1)}\left(\frac{\xi}{h}\right)\right|h^{-\alpha_1}|\partialxi^{\alpha_2}\japxih^{-k+1}| \nonumber \\
		& \times & \int_0^{|x|}\chi_\xi(y)\langle y\rangle^{-\frac{p-k}{p-1}\sigma}\sum_{j=1}^{\alpha_3}\frac{1}{j!}\left|\psi^{(j)}\left(\frac{\langle y\rangle}{\japxi^{p-1}}\right)\right| \nonumber \\
		& \times & \sum_{\gamma_1+\cdots+\gamma_j=\alpha_3}\frac{\alpha_3!}{\gamma_1!\cdots\gamma_j!}\prod_{\ell=1}^j\partialxi^{\gamma_\ell}\japxih^{1-p}\langle y\rangle dy  
		\end{eqnarray*}
	Then we have:
	\begin{eqnarray*}
		|\partialxi^\alpha\lambda_{p-k}(x,\xi)|
		& \leq & M_{p-k}\sum_{\alpha_1+\alpha_2+\alpha_3=\alpha}\frac{\alpha!}{\alpha_1!\alpha_2!\alpha_3!}C_\omega^{\alpha_1+1}\alpha_1!^\mu\japxih^{-\alpha_1}\langle R_{a_p}\rangle^{\alpha_1}C^{\alpha_2}\alpha_2!\japxih^{1-k-\alpha_2} \nonumber \\
		& \times & \int_0^{|x|}\chi_\xi(y)\langle y\rangle^{-\frac{p-k}{p-1}\sigma}\sum_{j=1}^{\alpha_3}\frac{1}{j!}C_\psi^{j+1} \nonumber \\
		& \times & \sum_{\gamma_1+\cdots+\gamma_j=\alpha_3}\frac{\alpha_3!}{\gamma_1!\cdots\gamma_j!}\langle y\rangle\prod_{\ell=1}^j C^{\gamma_\ell}\gamma_\ell!\japxih^{1-p-\gamma_\ell}dy \nonumber \\
		& \leq & M_{p-k}\sum_{\alpha_1+\alpha_2+\alpha_3=\alpha}\frac{\alpha!}{\alpha_1!\alpha_2!\alpha_3!}C_\omega^{\alpha_1+1}\alpha_1!^\mu\japxih^{-\alpha_1}\langle R_{a_p}\rangle^{\alpha_1}C^{\alpha_2}\alpha_2!\japxih^{1-k}\japxih^{-\alpha_2} \nonumber \\
		& \times & \int_0^{|x|}\langle y\rangle^{-\frac{p-k}{p-1}\sigma}\sum_{j=1}^{\alpha_3}\frac{1}{j!}C_\psi^{\alpha_3+1}\alpha_3!^\mu\sum_{\gamma_1+\cdots+\gamma_j=\alpha_3}\frac{\alpha_3}{\gamma_1!\cdots\gamma_j!}C^{\alpha_3}\gamma_1!\cdots\gamma_j!\japxih^{-\alpha_3} \nonumber \\
		& \leq & M_{p-k}C_{\omega,\psi,R_{a_p}}^{\alpha+1}\alpha!^\mu\japxih^{-\alpha}\japxih^{1-k}\int_0^{|x|}\chi_\xi(y)\langle y\rangle^{-\frac{p-k}{p-1}\sigma}dy \nonumber \\
		& \leq & \frac{M_{p-k}}{1-\frac{p-k}{p-1}\sigma}C_{\omega,\psi,R_{a_p}}^{\alpha+1}\japxih^{-\alpha}\japxih^{(p-k)(1-\sigma)} \nonumber \\
		& = & \frac{M_{p-k}}{1-\frac{p-k}{p-1}\sigma}C_{\omega,\psi,R_{a_p}}^{\alpha+1}\japxih^{(p-k)(1-\sigma)-\alpha}, \nonumber
	\end{eqnarray*}
	and we get (ii). To obtain (iii), we just need to observe that
	$$
	\int_0^{|x|}\chi_\xi(y)\langle y\rangle^{-\frac{p-k}{p-1}\sigma}dy\leq\japx^{1-\frac{p-k}{p-1}\sigma}.
	$$
	To obtain (iv), we use the fact that, on the support of $\psi(\langle y\rangle/\japxi_h^{p-1})$, we have $\japxih^{1-k}\leq\japx^{\frac{1-k}{p-1}}$, which implies
	$$
	\japxih^{1-k-\alpha}\japx^{1-\frac{p-k}{p-1}\sigma}\leq\japxih^{-\alpha}\japx^{\frac{p-k}{p-1}(1-\sigma)}.
	$$

	Now, by considering $\beta\geq1$, we have
	\begin{equation}\label{xderivative}
		\partialx^\beta\lambda_{p-k}(x,\xi)=M_{p-k}\omega\left(\frac{\xi}{h}\right)\japxih^{-k+1}\partialx^\beta\int_0^x\langle y\rangle^{-\frac{p-k}{p-1}\sigma}\psi\left(\frac{\langle y\rangle}{\japxih^{p-1}}\right)dy.
	\end{equation}
	Note that, we need to compute the derivative of order $\beta$ of the integral. If $\beta=1$, we have
	$$
	\partial_x\int_0^x\langle y\rangle^{-\frac{p-k}{p-1}\sigma}\psi\left(\frac{\langle y\rangle}{\japxih^{p-1}}\right)dy=\japx^{-\frac{p-k}{p-1}\sigma}\psi\left(\frac{\japx}{\japxih^{p-1}}\right),
	$$
	and then we can compute by using Leibniz formula
	\begin{eqnarray}
		\partialx^\beta\int_0^x\langle y\rangle^{-\frac{p-k}{p-1}\sigma}\psi\left(\frac{\langle y\rangle}{\japxih^{p-1}}\right)dy & = & \partialx^{\beta-1}\left[\japx^{-\frac{p-k}{p-1}\sigma}\psi\left(\frac{\japx}{\japxih^{p-1}}\right)\right] \nonumber \\
		& = & \sum_{\beta_1+\beta_2=\beta-1}\frac{(\beta-1)!}{\beta_1!\beta_2!}\partialx^{\beta_1}\japx^{-\frac{p-k}{p-1}\sigma}\partialx^{\beta_2}\psi\left(\frac{\japx}{\japxih^{p-1}}\right). \nonumber
	\end{eqnarray}
	Returning to \eqref{xderivative}, we can write
	\begin{eqnarray}
		\partialx^\beta\lambda_{p-k}(x,\xi) & = & M_{p-k}\omega\left(\frac{\xi}{h}\right)\japxih^{-k+1}\sum_{\beta_1+\beta_2=\beta-1}\frac{(\beta-1)!}{\beta_1!\beta_2!}\partialx^{\beta_1}\japx^{-\frac{p-k}{p-1}\sigma}\partialx^{\beta_2}\psi\left(\frac{\japx}{\japxih^{p-1}}\right) \nonumber \\
		& = & M_{p-k}\sum_{\beta_1+\beta_2=\beta-1}\frac{(\beta-1)!}{\beta_1!\beta_2!}\partialx^{\beta_1}\japx^{-\frac{p-k}{p-1}\sigma}\omega\left(\frac{\xi}{h}\right)\japxih^{-k+1}\partialx^{\beta_2}\psi\left(\frac{\langle x\rangle}{\japxih^{p-1}}\right) \nonumber  
	\end{eqnarray}
	and from this, it follows that
	\begin{eqnarray}
		\doublepartial\lambda_{p-k}(x,\xi) & = & M_{p-k}\sum_{\beta_1+\beta_2=\beta-1}\frac{(\beta-1)!}{\beta_1!\beta_2!}\partialx^{\beta_1}\japx^{-\frac{p-k}{p-1}\sigma}\partialxi^\alpha\left[\omega\left(\frac{\xi}{h}\right)\japxih^{-k+1}\partialx^{\beta_2}\psi\left(\frac{\langle x\rangle}{\japxih^{p-1}}\right)\right] \nonumber  \\
		& = & M_{p-k}\sum_{\beta_1+\beta_2=\beta-1}\frac{(\beta-1)!}{\beta_1!\beta_2!}\partialx^{\beta_1}\japx^{-\frac{p-k}{p-1}\sigma} \nonumber \\
		& \times & \sum_{\alpha_1+\alpha_2+\alpha_3=\alpha}\frac{\alpha!}{\alpha_1!\alpha_2!\alpha_3!}\partialxi^{\alpha_1}\omega\left(\frac{\xi}{h}\right)\partialxi^{\alpha_2}\japxih^{-k+1}\partialxi^{\alpha_3}\partialx^{\beta_2}\psi\left(\frac{\japx}{\japxih^{p-1}}\right) \nonumber \\
		& = & M_{p-k}\sum_{\beta_1+\beta_2=\beta-1}\frac{(\beta-1)!}{\beta_1!\beta_2!}\partialx^{\beta_1}\japx^{-\frac{p-k}{p-1}\sigma} \nonumber \\
		& \times & \sum_{\alpha_1+\alpha_2+\alpha_3=\alpha}\frac{\alpha!}{\alpha_1!\alpha_2!\alpha_3!}\omega^{(\alpha_1)}\left(\frac{\xi}{h}\right)h^{-\alpha_1}\partialxi^{\alpha_2}\japxih^{-k+1} \nonumber \\
		& \times & \sum_{j=1}^{\alpha_3+\beta_2}\frac{\psi^{(j)}\left(\frac{\japx}{\japxih^{p-1}}\right)}{j!}\sum_{\gamma_1+\cdots+\gamma_j=\alpha_3}\sum_{\delta_1+\cdots+\delta_j=\beta_2}\frac{\alpha_3!\beta_2!}{\gamma_1!\delta_1!\cdots\gamma_j!\delta_j!}\prod_{\ell=1}^j\partialx^{\delta_\ell}\japx\partialxi^{\gamma_\ell}\japxih^{1-p}. \nonumber
	\end{eqnarray}
	Now, we can estimate
	\begin{eqnarray}
		|\doublepartial\lambda_{p-k}(x,\xi)| & \leq & M_{p-k}\sum_{\beta_1+\beta_2=\beta-1}\frac{(\beta-1)!}{\beta_1!\beta_2!}|\partialx^{\beta_1}\japx^{-\frac{p-k}{p-1}\sigma}| \nonumber \\
		& \cdot &\hskip-0.5cm \sum_{\alpha_1+\alpha_2+\alpha_3=\alpha}\frac{\alpha!}{\alpha_1!\alpha_2!\alpha_3!}\left|\omega^{(\alpha_1)}\left(\frac{\xi}{h}\right)\right|h^{-\alpha_1}|\partialxi^{\alpha_2}\japxih^{-k+1}|\cdot \chi_\xi(x) \nonumber \\
		& \cdot &\hskip-0.5cm \sum_{j=1}^{\alpha_3+\beta_2}\frac{\left|\psi^{(j)}\left(\frac{\japx}{\japxih^{p-1}}\right)\right|}{j!}\sum_{\gamma_1+\cdots+\gamma_j=\alpha_3}\sum_{\delta_1+\cdots+\delta_j=\beta_2}\frac{\alpha_3!\beta_2!}{\gamma_1!\delta_1!\cdots\gamma_j!\delta_j!}\prod_{\ell=1}^j|\partialx^{\delta_\ell}\japx||\partialxi^{\gamma_\ell}\japxih^{1-p}| \nonumber \\
		& \leq & M_{p-k}\sum_{\beta_1+\beta_2=\beta-1}\frac{(\beta-1)!}{\beta_1!\beta_2!}C^{\beta_1}\beta_1!\japx^{-\frac{p-k}{p-1}\sigma-\beta_1} \nonumber \\
		& \cdot &\hskip-0.5cm \sum_{\alpha_1+\alpha_2+\alpha_3=\alpha}\frac{\alpha!}{\alpha_1!\alpha_2!\alpha_3!}C_\omega^{\alpha_1+1}\alpha_1!^\mu\japxih^{-\alpha_1}\langle R_{a_p}\rangle^{\alpha_1}C^{\alpha_1}\alpha_1!\japxih^{-k+1-\alpha_1} \nonumber \\
		& \cdot &\hskip-0.3cm \sum_{j=1}^{\alpha_3+\beta_2}\frac{1}{j!}C_{\psi}^{j+1}j!^\mu\sum_{\gamma_1+\cdots+\gamma_j=\alpha_3}\sum_{\delta_1+\cdots+\delta_j=\beta_2}\frac{\alpha_3!\beta_2!}{\gamma_1!\delta_1!\cdots\gamma_j!\delta_j!} \nonumber \\
		& \cdot & \prod_{\ell=1}^j C^{\delta_\ell}\delta_\ell!\japx^{1-\delta_\ell} C^{\gamma_\ell}\gamma_\ell!\japxih^{1-p-\gamma_\ell} \nonumber \\
		& \leq & M_{p-k}C_{\psi,\omega,R_{a_p}}^{\alpha+\beta+1}\alpha!^\mu(\beta-1)!^\mu\japxih^{1-k-\alpha}\japx^{-\frac{p-k}{p-1}\sigma-(\beta-1)}. \nonumber
	\end{eqnarray}
\end{proof}

\begin{remark} From Lemma \ref{estimatesforlambda} we  conclude that $\lambda_{p-k}\in \mathbf{SG}^{0,\frac{p-k}{p-1}(1-\sigma)}_\mu(\R^2)$ and $\lambda_{p-k}\in \mathbf{S}_\mu^{(p-k)(1-\sigma)}(\R^2)$, for each $k=1,...,p-1$. Hence
	\begin{equation}\label{function_Lambda}
		\Lambda=\sum_{k=1}^{p-1}\lambda_{p-k}\in \mathbf{SG}_\mu^{0,1-\sigma}(\R^2)\cap \mathbf{S}_\mu^{(p-1)(1-\sigma)}(\R^2).
	\end{equation}
\end{remark}

\subsection{Invertibility of the operator $e^\Lambda(x,D)$}

In this section we are going to construct the inverse operator of $e^\Lambda(x,D)$ in terms of the reverse operator $^R (e^{-\Lambda}(x,D))$. Notice that 
\begin{equation} \label{Lambda_reverse_estimate}
	|\partial_\xi^\alpha\Lambda(x,\xi)|\leq C_\Lambda^{\alpha+1}\alpha!^\mu \japxih^{(p-1)(1-\sigma)-\alpha}
\end{equation} and
\begin{equation}\label{Lambda_reverse_stronger_estimate}
	|\doublepartial\Lambda(x,\xi)|\leq C_\Lambda^{\alpha+\beta+1}(\alpha!\beta!)^\mu \japxih^{-\alpha},\quad \beta\geq 1.
\end{equation}
for some $C_\Lambda >0$. This estimate means that if at least an $x$-derivative falls on $\Lambda$, then we obtain a symbol of order $0$ ($<1/\kappa$). Taking into account  \eqref{Lambda_reverse_estimate} and \eqref{Lambda_reverse_stronger_estimate} the following result holds, cf. \cite[Lemma 3.4]{AAC3evolGevrey}.

\begin{lemma}\label{lemma_4_of_AA22}
	Let $\mu>1$. Then, for $h>0$ large enough, the operator $e^\Lambda(x,D)$ is invertible and its inverse is given by
	$$
	(e^\Lambda(x,D))^{-1} = \hskip2pt ^R (e^{-\Lambda}(x,D)) \circ (I+r(x,D))^{-1} = \hskip2pt ^R (e^{-\Lambda}(x,D)) \circ \sum_{j\geq0}(-r(x,D))^j,
	$$
	where $r=\tilde{r}+\bar{r}$ for some $\tilde{r} \in\mathbf{SG}_{\mu}^{-1,-\sigma}(\R^2)$ and $\bar{r}$ satisfying 
	\begin{equation}\label{aces_high}
	|\partial^{\alpha}_{\xi}\partial^{\beta}_{x}\bar{r}(x,\xi)| \leq C^{\alpha+\beta+1} (\alpha!\beta!)^{\kappa} e^{-c \{ \langle x \rangle^{\frac{1}{\kappa}} + \langle \xi \rangle^{\frac{1}{\kappa}} \} },
	\end{equation}
	with $\kappa > 2\mu -1$ and for some $C, c > 0$.
	Moreover, for every $N \in \N$ we have
	\begin{equation}\label{rasexpansion}
			\tilde{r} -\sum_{1 \leq \gamma \leq N}\frac{1}{\gamma!}\partialxi^\gamma\left(e^{\Lambda}D_x^\gamma e^{-\Lambda}\right) \in \mathbf{SG}_{\mu}^{-(N+1),-\sigma(N+1)}(\R^2),
	\end{equation}
	and the symbol of the operator $\sum(-r(x,D))^j$ is  of the form $q + q_{\infty}$, where $q \in \mathbf{SG}^{0,0}_{\mu}(\R^2)$ and $q_{\infty}$ satisfies \eqref{aces_high} for all $C, c >0$ and $\kappa > 2\mu-1$. 
\end{lemma}

Now we can prove some results about the symbol $r(x,\xi)$ of the operator $r(x,D)$ and on the corresponding Neumann series.

\begin{lemma}\label{lemma_asymptotic_expansion_of_r}
	The symbol $r(x,\xi)$ appearing in Lemma \ref{lemma_4_of_AA22} can be expressed as
	\begin{equation}\label{asymptotic_expansion_of_r}
		r=-\partialxi D_x\Lambda+b_{-2}+\cdots+b_{-(p-2)}+b_{-(p-1)}+b_{-p} 
	\end{equation}
	where $b_{-m}\in\mathbf{SG}_\mu^{-m,-\frac{p-m+1}{p-1}\sigma}(\R^2)$ depends only on $\lambda_{p-1},...,\lambda_{p-(m-1)}$, for $m=2,...,p-1$, and $b_{-p}$ is the sum of a symbol in $ \mathbf{SG}_{\mu}^{-p,-\frac1{p-1}\sigma}(\R^2)$ and of a symbol satisfying \eqref{aces_high} for some $C, c > 0$ and $\kappa = 2\mu-1$.
\end{lemma}

\begin{proof}
	From \eqref{rasexpansion} we have that
	\begin{equation}\label{asymptotic_of_r}
		r = -\partialxi D_x\Lambda + \sum_{\gamma=2}^{p-1} \frac{1}{\gamma!}\partialxi^\gamma\left(e^{\Lambda}D_x^\gamma e^{-\Lambda}\right) + q_{-p},
	\end{equation}
with $q_{-p} \in \mathbf{SG}_{\kappa}^{-p,-p\sigma}(\R^2)$.
	For $\gamma\geq2$, by Fa\`a di Bruno's formula we get
	$$
	e^\Lambda D_x^\gamma e^{-\Lambda} = \sum_{j=1}^\gamma\frac{(-1)^j}{j!}\sum_{\stackrel{\gamma_1+\cdots+\gamma_j=\gamma}{\gamma_\ell\geq1}}\frac{\gamma!}{\gamma_1!\cdots\gamma_j!}\prod_{\ell=1}^j D_x^{\gamma_\ell}\Lambda.
	$$
	Now, let us analyse $\partial_\xi^\gamma \left(\prod_{\ell=1}^j D_x^{\gamma_\ell}\Lambda\right)$. Since $\gamma \geq 2$ in the formula \eqref{asymptotic_of_r}, it is sufficient to prove that
	$$
	\prod_{\ell=1}^j D_x^{\gamma_\ell}\Lambda=\tilde{b}_{-2}+\cdots+\tilde{b}_{-(p-1)}+\tilde{b}_{-p},
	$$
	for some $\tilde{b}_{-m}\in\mathbf{SG}_\mu^{-m+2,-\frac{p-m+1}{p-1}\sigma}(\R^2)$, $m=2,...,p$, depending only on $\lambda_{p-1},...,\lambda_{p-(m-1)}$. Since $\Lambda=\lambda_{p-1}+\cdots+\lambda_1$, we have that $D_x^{\gamma_\ell}\Lambda=D_x^{\gamma_\ell}(\lambda_{p-1}+\cdots+\lambda_1)$, $\ell=1,...,j$, which implies that
	$$
	\prod_{\ell=1}^j D_x^{\gamma_\ell}\Lambda=(D_x^{\gamma_1}\lambda_{p-1}+\cdots+D_x^{\gamma_1}\lambda_1)\cdots(D_x^{\gamma_j}\lambda_{p-1}+\cdots+D_x^{\gamma_j}\lambda_1).
	$$
	By (v) of Lemma \ref{estimatesforlambda}, the above product has order zero with respect to $\xi$, and the only term with order exactly $0$ is
	$$
	D_x^{\gamma_1}\lambda_{p-1}\cdots D_x^{\gamma_j}\lambda_{p-1}\in\mathbf{SG}_\mu^{0,-\sigma}(\R^2).
	$$
	Similarly we notice that the only term of order exactly $-1$ with respect to $\xi$ is the sum of products of the form
	$$
	D_x^{\gamma_s}\lambda_{p-2}\prod_{\stackrel{\ell\le j}{\ell\ne s}}D_x^{\gamma_\ell}\lambda_{p-1}\in\mathbf{SG}_\mu^{-1,-\frac{p-2}{p-1}\sigma}(\R^2),
	$$
	which do not depend on $\lambda_{p-3},...,\lambda_1$. In general, we note that products containing at least one factor of the type $D_x^{\gamma_\ell}\lambda_{p-k}$ with $k\geq m$ have order at most $-m+1$, so the terms of order $-m+2$ with respect to $\xi$ cannot depend on $\lambda_{p-k}$ for $k\geq m$. Among these terms, the ones of highest order in $x$ are obviously those depending on $\lambda_{p-m+1}$ that is products of the form
	$$
	D_x^{\gamma_s}\lambda_{p-m+1}\prod_{\stackrel{\ell\le j}{\ell\ne s}}D_x^{\gamma_\ell}\lambda_{p-1}\in\mathbf{SG}_\mu^{-m+2,-\frac{p-m+1}{p-1}\sigma}(\R^2).
	$$
	This gives the assertion.
\end{proof}

By using Lemma \ref{lemma_asymptotic_expansion_of_r} we can prove the next result.

\begin{lemma}\label{lemma_structure_of_powers_of_r}
	For each $j\in\Bbb{N}$, $j\geq2$, the symbol $d_j$ of the operator $(-r(x,D))^j$ is of the form
	\begin{equation}\label{symbol_of_powers_of_r}
		d_j=b_{-j}^{[j]}+b_{-(j+1)}^{[j]}+\cdots+b_{-(p-1)}^{[j]}+b_{-p}^{[j]}
	\end{equation}
	for some $b_{-m}^{[j]}\in\mathbf{SG}_\mu^{-m,-\frac{jp-m}{p-1}\sigma}(\R^2)$ depending only on $\lambda_{p-1},...,\lambda_{p-m+j-1}$, for each $m=j,...,p-1$ and $b_{-p}^{[j]}$ is the sum of a symbol in  $\mathbf{SG}_\mu^{-p,-\frac{p(j-1)}{p-1}\sigma}(\R^2)$ and of a symbol satisfying \eqref{aces_high} for some $C, c >0$ and $\kappa = 2\mu-1$.
\end{lemma}

\begin{proof}
	We can argue by induction on $j$. For $j=2$, we have
	\begin{eqnarray}
		(-r(x,D))^2 & = & \sum_{k=1}^{p-1}\sum_{\ell=1}^{p-1}(\partialxi D_x\lambda_{p-k})(x,D)\circ(\partialxi D_x\lambda_{p-\ell})(x,D) \nonumber \\
		& - & \left(\sum_{k=1}^{p-1}\partialxi D_x\lambda_{p-k}\right)(x,D)\circ\left(\sum_{\ell=2}^p b_{-\ell}\right)(x,D) \nonumber \\
		& - & \left(\sum_{\ell=2}^p b_{-\ell}\right)(x,D)\circ\left(\sum_{k=1}^{p-1}\partialxi D_x\lambda_{p-k}\right)(x,D) \nonumber \\
		& + & \sum_{s=2}^p\sum_{\ell=2}^p b_{-s}(x,D)\circ b_{-\ell}(x,D) \nonumber \\
		& =: & \mathtt{F}+\mathtt{G}+\mathtt{H}+\mathtt{L}. \nonumber
	\end{eqnarray}
	Consider
	$$
	\mathtt{F}=\sum_{k=1}^{p-1}\sum_{\ell=1}^{p-1}(\partialxi D_x\lambda_{p-k})(x,D)\circ(\partialxi D_x\lambda_{p-\ell})(x,D).
	$$
	We immediately notice that, since the order of each term of the sum $F$ with respect to $\xi$ is $-k-\ell$, then the terms with order exactly $-m$ w.r.t. $\xi$ cannot depend on $\lambda_{p-k}$ for $k\geq m$ or $\lambda_{p-\ell}$ for $\ell\geq m$, so they are of the form $(\partialxi D_x\lambda_{p-k})(x,D)\circ(\partialxi D_x\lambda_{p-\ell})(x,D)$ for $k+\ell=m$. Hence their symbols belong to $\mathbf{SG}_\mu^{-m,-\frac{2p-m}{p-1}\sigma}(\R^2)$. \\ Concerning
	$$
	\mathtt{G}=\left(\sum_{k=1}^{p-1}\partialxi D_x\lambda_{p-k}\right)(x,D)\circ\left(\sum_{\ell=2}^p b_{-\ell}\right)(x,D),
	$$
	we notice that for $k\geq m$ the compositions of $(\partialxi D_x\lambda_{p-k})(x,D)$ with $\sum_{\ell=2}^p b_{-\ell}(x,D)$ give operators with order less than or equal to $-m-2$ w.r.t. $\xi$. Hence, the terms of order exactly $-m$ in this sum do not depend on $\lambda_{p-2}$ for $s\geq m$. Moreover, these terms are of the form $\partialxi D_x \lambda_{p-s}(x,D)\circ b_{-\ell}(x,D)$ with $\ell=m-s$, hence they belong to $\mathbf{SG}_\mu^{-m,-\frac{2p-m+1}{p-1}}(\R^2)$. The term $\mathtt{H}$ can be treated in the same way. \\ Finally, let us consider the term
	$$
	\mathtt{L}=\sum_{s=2}^p\sum_{\ell=2}^p b_{-s}(x,D)\circ b_{-\ell}(x,D).
	$$ The terms depending on $\lambda_{p-k}$ with $k\geq m$ appear only in compositions of the form $b_{-s}(x,D) \circ b_{-\ell}(x,D)$ with $\ell > m$ and/or $s > m$, which have order $-s-\ell < -m-2$. Moreover, every term of order exactly $-m$ is obtained as a product of the form $b_{-s}(x,D) \circ b_{-\ell}(x,D)$ for $s+\ell=m$ and it decays like $\japx^{-\frac{2p-m+2}{p-1}\sigma}$. In conclusion we have obtained the assertion for $j=2$. \\ The inductive step follows from similar considerations. Assume now that the statement is true for $j\leq N$ and consider
	\begin{eqnarray}
		(-r(x,D))^{N+1} & = & (-r(x,D))\circ(-r(x,D))^{N} \nonumber \\
		& = & \mathbf{op}(-\partialxi D_x\lambda_{p-1}-\cdots\partialxi D_x\lambda_1+b_{-2}+\cdots+b_{-p})\circ\mathbf{op}(b_{-N}^{[N]}+\cdots+b_{-p}^{[N]}). \nonumber 
	\end{eqnarray}
	Notice that the only term in the composition with order exactly $-N-1$ w.r.t. $\xi$ is given by$(-\partialxi D_x\lambda_{p-1})(x,D)\circ b_{-N}^{[N]}(x,D)$ whose symbol depends only on $\lambda_{p-1}$ and belongs to the class $\mathbf{SG}_\mu^{-N-1,-\frac{(N+1)p-N-1}{p-1}\sigma}(\R^2)$. The only term with order exactly $-N-2$ w.r.t. $\xi$ is given by
	$$
	(-\partialxi D_x\lambda_{p-1})(x,D)\circ b_{-N-1}^{[N]}(x,D)+(-\partialxi D_x\lambda_{p-2})(x,D)\circ b_{-N}^{[N]}(x,D)
	$$
	whose symbol depends only on $\lambda_{p-1}$, $\lambda_{p-2}$ and belongs to $\mathbf{SG}_{\mu}^{-N-2,-\frac{(N+1)p-N-2}{p-1}}(\R^2)$, and so on.
\end{proof}

From Lemmas \ref{lemma_asymptotic_expansion_of_r}, \ref{lemma_structure_of_powers_of_r} the next result follows.

\begin{proposition}\label{proposition_inverse_of_e^lambda}
	Let $\mu>1$. For $h$ large enough, say $h> h_0$, the operator $e^\Lambda(x,D)$ is invertible and its inverse is given by
	\begin{equation}\label{equation_inverse_of_e_power_tilde_Lambda_in_a_precise_way}
		(e^\Lambda(x,D))^{-1}=\hskip2pt ^R (e^{-\Lambda}(x,D))\circ\mathbf{op}(1-i\partialxi \partialx\Lambda+q_{-2}+\cdots+q_{-(p-1)}+q_{-p}),
	\end{equation}
	where $q_{-m}\in\mathbf{SG}_\mu^{-m,-\frac{2p-m}{p-1}\sigma}(\R^2)$ for $m=2,...,p-1$, depend only on $\lambda_{p-1},...,\lambda_{p-(m-1)}$ and $q_{-p}$ is the sum of a symbol in $\mathbf{SG}_\mu^{-p, -\frac{p}{p-1}\sigma}(\R^2)$ and a symbol satisfying \eqref{aces_high} for all $C, c >0$ and $\kappa >2\mu-1$.
\end{proposition}

\section{Conjugation of the operator $iP$}\label{conjugation}

In this section the goal is to perform the conjugation of the operator $iP$ by 
$$
\boxed{Q_{\Lambda,K,\rho^\prime}(t,x,D)=e^{\Lambda_{K,\rho^\prime}}(t,D)\circ e^\Lambda(x,D)},
$$ where
$$
\Lambda(x,\xi) = \sum_{k=1}^{p-1} \lambda_{p-k}(x,\xi) \qquad \text{and} \qquad \Lambda_{K,\rho^\prime}(t,\xi) = K(T-t)\japxih^{(p-1)(1-\sigma)} + \rho^\prime\japxih^{1/\theta}.
$$
Since the inverse of $e^\Lambda(x,D)$ is $\hskip2pt ^R(e^{-\Lambda}(x,D))\circ\sum_{j\geq0}(-r(x,D))^j$, it is necessary to work with compositions of the form $e^\Lambda(x,D)\circ p(x,D)\circ \hskip2pt ^R(e^{-\Lambda}(x,D))$, where $p$ is an operator with symbol of finite order. 

\subsection{Conjugation of $iP$ by $e^\Lambda$}

Before starting with the conjugation, let us note that by Lemma \ref{estimatesforlambda} the function $\Lambda$ satisfies
$$
|\doublepartial\Lambda(x,\xi)| \leq \left\lbrace\begin{array}{l}
	C_\Lambda^{\alpha+1}(\alpha!)^\mu\japxih^{(p-1)(1-\sigma)-\alpha}, \quad\beta =0, \\ 
	C_\Lambda^{\alpha+\beta+1}(\alpha!\beta!)^\mu\japxih^{-\alpha},\quad\quad\,\,\,\,\,\beta\geq1,
\end{array} \right.
$$
where $C_\Lambda$ is a constant depending on $M_{p-1},...,M_1,C_\omega,C_\psi,\mu,\sigma$. By the assumption $(p-1)(1-\sigma)<\frac{1}{\theta}$ it follows that
\begin{eqnarray}
	|\doublepartial\Lambda(x,\xi)| & \leq & C_\Lambda^{\alpha+\beta+1}(\alpha!\beta!)^\mu\japxih^{(p-1)(1-\sigma)-\alpha} \nonumber \\
	& = & C_\Lambda^{\alpha+\beta+1}(\alpha!\beta!)^\mu\japxih^{\frac{1}{\theta}-\alpha}\japxih^{(p-1)(1-\sigma)-\frac{1}{\theta}} \nonumber \\
	& \leq & C_\Lambda h^{(p-1)(1-\sigma)-\frac{1}{\theta}}C_\Lambda^{\alpha+\beta}(\alpha!\beta!)^\mu\japxih^{\frac{1}{\theta}-\alpha}, \nonumber
\end{eqnarray}
therefore
$$
\rho(\Lambda):=h^{(p-1)(1-\sigma)-\frac{1}{\theta}}C_\Lambda
$$
can be taken as small as we want, provided that $h$ can be taken suitably large; thus, for $h$ large enough, we are able to compute $e^\Lambda(x,D)\circ(iP)\circ(e^\Lambda(x,D))^{-1}$ by using Theorem \ref{theorem_conjugation}.

	We are now going to conjugate each term of the operator $iP$, which is given by
	\begin{equation}
		iP(t,x,D_t,D_x) = \partial_t + ia_p(t)D_x^p + \sum_{j=1}^p ia_{p-j}(t,x)D_x^{p-j}.
	\end{equation}
 The next items are dedicated to this purpose.

\begin{itemize}

	\item \textbf{Conjugation of $\partial_t$.} Since $\Lambda$ does not depend on $t$, the conjugation of $\partial_t$ is given by
	$$
	e^\Lambda(x,D) \circ \partial_t \circ (e^\Lambda(x,D))^{-1} = \partial_t.
	$$

	\item \textbf{Conjugation of $ia_p(t)D_x^p$.} First of all, we will treat only the conjugation of $iD_x^p$ by $e^\Lambda(x,D)$ (since $a_p$ does not depend on $x$). By Theorem \ref{theorem_conjugation} it follows that
		\begin{eqnarray}\label{conjugation_of_D^p}
			e^\Lambda(x,D) \circ iD_x^p \circ \hskip2pt ^R (e^{-\Lambda}(x,D))&=&iD_x^p+\mathbf{op}\left(i\sum_{1\leq\alpha\leq p-1}\frac{1}{\alpha!}\partialxi^\alpha\left(e^\Lambda \xi^p D_x^\alpha e^{-\Lambda}\right)\right)
\nonumber\\
&+&r_0(x,D)+r_\infty(x,D).
		\end{eqnarray}
	Note that we can write
	\begin{equation}\label{xi^p_conjecture}
		\sum_{1\leq\alpha\leq p-1} \frac{1}{\alpha!} \partialxi^\alpha \left(e^\Lambda \xi^p D_x^\alpha e^{-\Lambda}\right) = \partialxi \left(\xi^p D_x(-\Lambda)\right) + c_{-2} + c_{-3} + \cdots + c_{-(p-1)} + c_{-p},
	\end{equation}
	where $c_{-m} \in \mathbf{SG}_\mu^{p-m,-\frac{p-m+1}{p-1}\sigma}(\R^2)$ and it depends only on $\lambda_{p-1}, ..., \lambda_{p-m+1}$, for each $m=2,...,p$. Now we can substitute \eqref{xi^p_conjecture} in \eqref{conjugation_of_D^p} to obtain
		\begin{eqnarray}\label{conjugation_of_iD^p}
			e^\Lambda(x,D) \circ iD_x^p \circ \hskip2pt ^R (e^{-\Lambda}(x,D)) & = & iD_x^p + i\mathbf{op}(\partialxi\{\xi^p D_x(-\Lambda)\}) 
\nonumber\\
&+& \ds\sum_{j=2}^p ic_{-j}(x,D) +  r_0(x,D) + r_\infty(x,D).
		\end{eqnarray}
	 From Proposition \ref{proposition_inverse_of_e^lambda}, we have that the Neumann series is given by
	\begin{equation}\label{neumann_series}
		\sum_{j\geq0}(-r(x,D))^j=\mathbf{op}(1-i\partialxi\partialx\Lambda+q_{-2}+\cdots+q_{-(p-1)}+q_{-p}).
	\end{equation}
	Since we have to perform a composition between the two operators given in \eqref{conjugation_of_iD^p} and \eqref{neumann_series}, let us analyse what happens with the composition
	$$
	\left(iD_x^p + i\mathbf{op}(\partialxi\{\xi^p D_x(-\Lambda)\}) +  \ds\sum_{j=2}^p ic_{-j}(x,D)\right)\circ \mathbf{op}(1 - i\partialxi\partialx\Lambda + \ds\sum_{h=2}^p q_{-h} ).
	$$
	We can rewrite it as
	\begin{eqnarray}\label{the_composition_1}
		&& \left(iD_x^p+ i\mathbf{op}(\partialxi\{\xi^p D_x(-\Lambda)\}) +  \ds\sum_{j=2}^p ic_{-j}(x,D)\right)\circ \left(I + \mathbf{op}(\partialxi D_x\Lambda) + \ds\sum_{h=2}^p q_{-h} (x,D)\right)\nonumber \\
		 &&\hskip+1cm=  iD_x^p + i\mathbf{op}(\partialxi\{\xi^p D_x(-\Lambda)\})+\ds\sum_{j=2}^p ic_{-j}(x,D)+ iD_x^p \circ \mathbf{op}(\partialxi D_x\Lambda) \nonumber \\
		 &&\hskip+1cm + i\mathbf{op}(\partialxi\{\xi^p D_x(-\Lambda)\}) \circ \mathbf{op}(\partialxi D_x\Lambda) +  \ds\sum_{j=2}^p ic_{-j}(x,D) \circ (\partialxi D_x\Lambda)(x,D)+  iD_x^p \circ \ds\sum_{h=2}^p q_{-h} (x,D)\nonumber \\
		 &&\hskip+1cm+  i\mathbf{op}(\partialxi\{\xi^p D_x(-\Lambda)\}) \circ \ds\sum_{h=2}^p q_{-h} (x,D)+ \ds\sum_{j=2}^p ic_{-j}(x,D) \circ \ds\sum_{h=2}^p q_{-h}(x,D)  \nonumber \\
		&&\hskip+1cm=  iD_x^p + i\mathbf{op}\left(p\xi^{p-1}D_x(-\Lambda) + \cancel{\xi^p \partialxi D_x(-\Lambda)}\right) +\ds\sum_{j=2}^p ic_{-j}(x,D)+ \cancel{i\mathbf{op}(\xi^p\partialxi D_x \Lambda)}\nonumber \\
		 &&\hskip+1cm + i\mathbf{op}\left( \sum_{\gamma=1}^p \frac{1}{\gamma!} \partialxi^\gamma \xi^p \partialxi D_x^{\gamma+1} \Lambda \right) + i\mathbf{op}(\partialxi\{\xi^p D_x(-\Lambda)\}) \circ \mathbf{op}(\partialxi D_x\Lambda) \nonumber \\
		 &&\hskip+1cm+ \ds\sum_{j=2}^p ic_{-j}(x,D) \circ (\partialxi D_x\Lambda)(x,D) +  iD_x^p \circ \ds\sum_{h=2}^p q_{-h} (x,D)  \nonumber \\
		&&\hskip+1.3cm  + i\mathbf{op}(\partialxi\{\xi^p D_x(-\Lambda)\}) \circ \ds\sum_{h=2}^p q_{-h} (x,D)+ \ds\sum_{j=2}^p ic_{-j}(x,D) \circ \ds\sum_{h=2}^p q_{-h} (x,D). 
	\end{eqnarray}
	Next, let us make some remarks about order and dependence  on the constants $M_{p-j}$ of some terms.
	Let us first look at the term 
	$$
	i \mathbf{op} \left( \sum_{\gamma=1}^p \frac{1}{\gamma!} \partialxi^\gamma \xi^p \partialxi D_x^{\gamma+1} \Lambda \right).
	$$
	By definition of $\Lambda$, we have for each $\gamma \geq 1$ 
	\begin{equation}\label{t1}
	\partialxi^\gamma \xi^p D_x^{\gamma+1} \partialxi \Lambda = \ds\sum_{k=1}^{p-1}\partialxi^\gamma \xi^p D_x^{\gamma+1} \partialxi \lambda_{p-k}
	\end{equation}
	For each $k = 1, ..., p-1$, we obtain that $\partialxi^\gamma \xi^p D_x^{\gamma+1} \partialxi \lambda_{p-k}$ has order 
	$$
	p-k-\gamma \leq p-k-1\ \text{w.r.t.}\ \xi
	$$
	and
	$$
	- \frac{p-k}{p-1} \sigma - \gamma \leq - \frac{p-k}{p-1} \sigma - 1\ \text{w.r.t.}\ x.
	$$
	Notice that the highest order with respect to $\xi$ in \eqref{t1} is $p-2$.
	
	\noindent The term $i \mathbf{op}( \partialxi \{ \xi^p D_x(-\Lambda) \} ) \circ \mathbf{op}( \partialxi D_x \Lambda )$ can be written, by definition of $\Lambda$, as
	\begin{equation}\label{t2}
	i \mathbf{op}( \partialxi \{ \xi^p D_x(-\Lambda) \} ) \circ \mathbf{op}( \partialxi D_x \Lambda ) = -i \sum_{\ell=1}^{p-1} \sum_{s=1}^{p-1} \mathbf{op}( \partialxi \{ \xi^p D_x \lambda_{p-\ell} \} ) \circ \mathbf{op}( \partialxi D_x \lambda_{p-s} ).
	\end{equation}
	Note that, for $\ell = 1, ..., p-1$, the symbol $\partialxi \{ \xi^p D_x \lambda_{p-\ell} \}$ has order $p-\ell$ w.r.t. $\xi$ and $-\frac{p-\ell}{p-1} \sigma$ w.r.t. $x$ and, for $s = 1, ..., p-1$, the symbol $\partialxi D_x \lambda_{p-s}$ has order $-s$ w.r.t. $\xi$ and $-\frac{p-s}{p-1} \sigma$ w.r.t. $x$. Hence, the operator $\mathbf{op}(\partialxi\{\xi^p D_x \lambda_{p-\ell}\}) \circ \mathbf{op}(\partialxi D_x \lambda_{p-s})$ has order
	$$
	p-(\ell+s)\ \text{w.r.t.}\ \xi
	$$
	and
	$$
	-\frac{2p-(\ell+s)}{p-1} \sigma \ \text{w.r.t.}\ x,
	$$
	for each $\ell=1,...,p-1$ and $s=1,...,p-1$. If we fix $j=1,...,p-1$ and we consider the terms of order $p-j$ with respect to $\xi$ in \eqref{t2} (that is, the terms with $\ell+s=j$) we see that they cannot depend on $\lambda_{p-k}$ with $k\geq j$, because otherwise the order would be strictly less than $p-j$. With respect to $x$, the order is
	$$
	- \frac{2p-j}{p-1} \sigma \leq - \frac{p-j+1}{p-1} \sigma,
	$$
	because $2p - j \geq p - j + 1$, for $p > 1$. Again, the highest order with respect to $\xi$ in \eqref{t2} is $p-2$, since $\ell$ and $s$ run through the set $\{1, ..., p-1\}$. \\
	 The symbol of the term $iD_x^p \circ q_{-h}(x,D)$, $h=2,\ldots, p-1$, behaves like $\japxi^{p-h}\japx^{-\frac{2p-h}{p-1}\sigma} \leq \japxi^{p-h} \japx^{-\frac{p-h}{p-1}\sigma}$ and depends only on $\lambda_{p-1}, \lambda_{p-2}, \ldots, \lambda_{p-(h-1)}$.
	
	Concerning the term
	$$
	i\mathbf{op}(\partialxi\{\xi^p D_x(-\Lambda)\}) \circ [q_{-2}(x,D)+\cdots+q_{-p}(x,D)],
	$$
	we see that the first term in this composition can be rewritten as
	$$
	-i\mathbf{op}(\partialxi\{\xi^p D_x(\lambda_{p-1}+\cdots+\lambda_{1})\} = -i\mathbf{op}(\partialxi\{\xi^p D_x\lambda_{p-1}\}+\cdots+\partialxi\{\xi^p D_x\lambda_1\}).
	$$
	Hence the composition is a sum of terms with symbols of the form $\partialxi\{\xi^p D_x \lambda_{p-k}\}\cdot q_{-\ell}$, for $k=1, ..., p-1$ and $\ell = 2, ..., p$, whose order is $p-(k+\ell)$ depending only on $\lambda_{p-1}, ..., \lambda_{p-k-\ell+1}$. Therefore, the highest order here is $p-3$.

	The other terms are all of negative order, then there is no need to make the same analysis as above. Summarizing: the conjugation $e^\Lambda(x,D) \circ (ia_p(t)D_x^p) \circ (e^\Lambda(x,D))^{-1}$ can be expressed as
	\begin{eqnarray}\label{leading_term_conjugation}
		e^\Lambda(x,D) \circ (ia_p(t)D_x^p) \circ (e^\Lambda(x,D))^{-1} & = & ia_p(t)D_x^p - \mathbf{op}\left(\partialxi a_p(t)\xi^p \partialx\Lambda + a_p(t)\sum_{m=2}^{p-1} \mathtt{c}_{p-m}\right) \nonumber \\
		& + & r_0(x,D) + r_\infty(x,D),
	\end{eqnarray}
		with $r_0$ of order zero, $r_\infty$ satisfying an estimate of the form \eqref{regularizingestimate}, $\mathtt{c}_{p-m}$ has order $p-m$ with respect to $\xi$ and $-\frac{p-m+1}{p-1}\sigma$ with respect to $x$ and depends only on $\lambda_{p-1}, ..., \lambda_{p-m+1}$, for each $m=2, ..., p-1$. 
	
	\item \textbf{Conjugation of $ia_{p-j}(t,x)D_x^{p-j}, j=1,...,p-1$.} For some $N \in \Bbb{N}$ to be chosen later, by Theorem \ref{theorem_conjugation} it follows that
	\begin{small}
		\begin{eqnarray}\label{conjugation_of_a_(p-j)D^(p-j)}
			& & e^\Lambda(x,D) \circ a_{p-j}(t,x) D_x^{p-j} \circ \hskip2pt ^R (e^{-\Lambda}(x,D))  \nonumber \\
			& = & a_{p-j}(t,x) D_x^{p-j}  +  \mathbf{op} \left(\sum_{1 \leq \alpha + \beta \leq N} \frac{1}{\alpha! \beta!} \partialxi^\alpha \left(\partialxi^\beta e^\Lambda \cdot D_x^\beta a_{p-j}(t,x) \xi^{p-j} \cdot D_x^\alpha e^{-\Lambda}\right) \right) \nonumber \\
			& + & r_0(x,D) + r_\infty(x,D).
		\end{eqnarray}
	\end{small}
	For the next computations, we omit $(x,\xi)$ and $(x,D)$ in the notation for simplicity. Our goal here is to analyse what happens with the terms in
	\begin{small}
		\begin{equation}\label{terms_sum_conjugation_of_a_(p-1)D^(p-1)}
			\sum_{1 \leq \alpha + \beta \leq N} \frac{1}{\alpha! \beta!} \partialxi^\alpha \left(\partialxi^\beta e^\Lambda \cdot D_x^\beta a_{p-j} \xi^{p-j} \cdot D_x^\alpha e^{-\Lambda}\right).
		\end{equation}
	\end{small}
	By Fa\`a di Bruno formula, the terms $\partialxi^\beta e^\Lambda$ and $D_x^\alpha e^{-\Lambda}$ become
	$$
	\partialxi^\beta e^\Lambda = \sum_{\mathtt{j}=1}^\beta \frac{1}{\mathtt{j}!} e^{\Lambda} \sum_{\stackrel{\beta_1 + \cdots + \beta_\mathtt{j} = \beta}{\beta_\mathtt{k} \geq 1}} \frac{\beta!}{\beta_1! \cdots \beta_\mathtt{j}!} \prod_{\mathtt{k} = 1}^\mathtt{j} \partialxi^{\beta_\mathtt{k}} \Lambda 
	$$
	and
	$$
	D_x^\alpha e^{-\Lambda} = \sum_{\ell = 1}^\alpha \frac{(-1)^\ell}{\ell!} e^{-\Lambda} \sum_{\stackrel{\alpha_1 + \cdots + \alpha_\ell = \alpha}{\alpha_\mathtt{m} \geq 1}} \frac{\alpha!}{\alpha_1! \cdots \alpha_\ell!} \prod_{\mathtt{m} = 1}^\ell D_x^{\alpha_\mathtt{m}} \Lambda.
	$$
	By fixing $\alpha + \beta$, the general term in \eqref{terms_sum_conjugation_of_a_(p-1)D^(p-1)} will be a sum of terms which will be of the form
	\begin{equation}\label{I_alphabeta}
		I_{(\alpha,\beta)}^{[j]} := \partialxi^\alpha\left(D_x^\beta a_{p-j} \xi^{p-j} \cdot \partialxi^{\beta_1}\lambda_{p-k_1} \cdots \partialxi^{\beta_\mathtt{r}}\lambda_{p-k_\mathtt{r}} \cdot D_x^{\alpha_1}\lambda_{p-\tilde{k}_1} \cdots D_x^{\alpha_\mathtt{s}}\lambda_{p-\tilde{k}_\mathtt{s}}\right),
	\end{equation}
	where
	$$
	k_1, ..., k_\mathtt{r}, \tilde{k}_1, ..., \tilde{k}_\mathtt{s} \in \{1, ..., p-1\} \qquad \text{and} \qquad 1 \leq \mathtt{r} \leq \beta, \quad 1 \leq \mathtt{s} \leq \alpha. 
	$$
	Now, we want to show that if $\lambda_{p-k}$ appears in $I_{(\alpha,\beta)}^{[j]}$ with $k \geq m$, then $I_{(\alpha,\beta)}^{[j]}$ is of order at most $p - m -1$ with respect to $\xi$.

	\textbf{Case 1}: Let suppose that in \eqref{I_alphabeta} there is a term of the type $D_x^{\alpha_1}\lambda_{p-\tilde{k}_1}$ with $\tilde{k}_1 \geq m$ and $m\geq2$.

	By using (v) of Lemma \ref{estimatesforlambda} and the fact that $\tilde{k}_1 \geq m$ and each $\tilde{k}_\mathtt{a} \leq p-1$, for all $1 \leq \mathtt{a} \leq \mathtt{s}$, it follows that
	\begin{equation}\label{case_1_equation_1}
		\left| \left(D_x^{\alpha_1}\lambda_{p-\tilde{k}_1} \cdots D_x^{\alpha_\mathtt{s}}\lambda_{p-\tilde{k}_\mathtt{s}}\right) (x,\xi) \right| \leq C^{\alpha+1} (\alpha!)^\mu \japxi^{-m+1} \japx ^{-\alpha+\mathtt{s}-\frac{\mathtt{s}\sigma}{p-1}}.
	\end{equation}
	Now, from (iv) of Lemma \ref{estimatesforlambda} and the fact that $k_\mathtt{a}\geq1$, for all $1 \leq \mathtt{a} \leq \mathtt{r}$, we obtain the estimate
	\begin{equation}\label{case_1_equation_2}
		\left| \left( \partialxi^{\beta_1}\lambda_{p-k_1} \cdots \partialxi^{\beta_\mathtt{r}}\lambda_{p-k_\mathtt{r}} \right)(x,\xi) \right| \leq C^{\beta+1} (\beta!)^\mu \japxi^{-\beta} \japx^{\mathtt{r}(1-\sigma)}.
	\end{equation}
	By hypothesis, we have that
	\begin{equation}\label{case_1_equation_3}
		\left| D_x^\beta a_{p-j}(t,x)\xi^{p-j} \right| \leq C^{\beta+1} (\beta!)^{\theta_0} \japxi^{p-j} \japx^{-\frac{p-j}{p-1}\sigma-\beta}.
	\end{equation}
	From the definition \eqref{I_alphabeta} of $ I_{(\alpha,\beta)}^{[j]}$, \eqref{case_1_equation_1}, \eqref{case_1_equation_2} and \eqref{case_1_equation_3} we get
	\begin{equation}\label{case_1_equation_4}
	| I_{(\alpha,\beta)}^{[j]}(x,\xi) | \leq C_{\alpha,\beta} \japxi^{-m+1-\beta+p-j-\alpha} \japx^{-\alpha+\mathtt{s}-\frac{\mathtt{s}\sigma}{p-1}+\mathtt{r}(1-\sigma)-\frac{p-j}{p-1}\sigma-\beta}.
	\end{equation}
	The exponent of $\japxi$ is $p-m-j+1-\beta-\alpha\leq p-j-m$ since $\alpha+\beta\geq 1$. Regarding the exponent of $\japx$, note that
	$$
	-\alpha+\mathtt{s}\underbrace{-\frac{\mathtt{s}\sigma}{p-1}}_{<0}+\underbrace{\mathtt{r}(1-\sigma)}_{<\mathtt{r}}-\frac{p-j}{p-1}\sigma-\beta \leq -\frac{p-j}{p-1}\sigma+\underbrace{(\mathtt{s}-\alpha)}_{\mathtt{s} \leq \alpha}+\underbrace{(\mathtt{r}-\beta)}_{\mathtt{r} \leq \beta} \leq -\frac{p-j}{p-1}\sigma.
	$$
	Hence
	\begin{equation}\label{case_1_equation_5}
	| I_{(\alpha,\beta)}^{[j]}(x,\xi) | \leq C_{\alpha,\beta} \japxi^{p-m-j+1-\alpha-\beta} \japx^{-\frac{p-j}{p-1}\sigma},
	\end{equation}
	which implies that $I_{(\alpha,\beta)}^{[j]} \in \mathbf{SG}_\mu^{p-m-j+1-\alpha-\beta,-\frac{p-j}{p-1}\sigma}(\R^2)$.

	\textbf{Case 2}: Assume that $\partialxi^{\beta_1}\lambda_{p-k_1}$, with $k_1 \geq m$, appears in $I^{[j]}_{(\alpha,\beta)}$.

	Here we need to present more details in the computations, because this case is a little bit trickier than the previous one. By using estimates (iii) and (iv) of Lemma \ref{estimatesforlambda} and $k_1 \geq m$,
	\begin{eqnarray}\label{case_2_equation_1}
		& & \left| \left( \partialxi^{\beta_1}\lambda_{p-k_1} \cdots \partialxi^{\beta_\mathtt{r}}\lambda_{p-k_\mathtt{r}} \right)(x,\xi) \right| \nonumber \\
		& \leq & C^{\beta+1} (\beta!)^\mu \underbrace{\japxi^{1-k_1-\beta_1} \japx^{1-\frac{p-k_1}{p-1}\sigma}}_{\text{from (iii)}} \underbrace{\japxi^{-\beta_2} \japx^{\frac{p-k_2}{p-1}(1-\sigma)} \cdots \japxi^{-\beta_\mathtt{r}} \japx^{\frac{p-k_\mathtt{r}}{p-1}(1-\sigma)}}_{\text{from (iv)}} \nonumber \\
		& \leq & C^{\beta+1} (\beta!)^\mu \japxi^{1-m-\beta} \japx^{1-\frac{p-k_1}{p-1}\sigma+\frac{p-k_2}{p-1}(1-\sigma)+\cdots+\frac{p-k_\mathtt{r}}{p-1}(1-\sigma)}.
	\end{eqnarray}
	By using (v) of Lemma \ref{estimatesforlambda} and $1\leq \tilde{k}_\mathtt{a} \leq p-1$, for each $1 \leq \mathtt{a} \leq \mathtt{s}$, we obtain
	\begin{eqnarray}\label{case_2_equation_3}
		& & \left| \left( D_x^{\alpha_1}\lambda_{p-\tilde{k}_1} \cdots D_x^{\alpha_\mathtt{s}}\lambda_{p-\tilde{k}_\mathtt{s}} \right)(x,\xi) \right| \nonumber \\
		& \leq & C^{\alpha+1} (\alpha!)^\mu \japxi^{1-\tilde{k}_1+1-\tilde{k}_2+\cdots+1-\tilde{k}_\mathtt{s}} \japx^{-\frac{p-\tilde{k}_1}{p-1}\sigma-\alpha_1+1-\cdots-\frac{p-\tilde{k}_\mathtt{s}}{p-1}\sigma-\alpha_\mathtt{s}+1} \nonumber \\
		& = & C^{\alpha+1} (\alpha!)^\mu \japxi^{\mathtt{s}\overbrace{-\tilde{k}_1-\cdots-\tilde{k}_\mathtt{s}}^{\text{all} \ \tilde{k}_\mathtt{a} \geq 1}} \japx^{-\alpha+\mathtt{s}-\frac{\mathtt{s}p\overbrace{-\tilde{k}_1-\cdots-\tilde{k}_\mathtt{s}}^{\text{all} \ \tilde{k}_\mathtt{a} \leq p-1}}{p-1}\sigma} \nonumber \\
		& \leq & C^{\alpha+1} (\alpha!)^\mu \japxi^{\mathtt{s}-\mathtt{s}} \japx^{-\alpha+\mathtt{s}\overbrace{-\frac{\mathtt{s}p-\mathtt{s}(p-1)}{p-1}\sigma}^{\leq 0}} \nonumber \\
		& \leq & C^{\alpha+1} (\alpha!)^\mu \japx^{-\alpha+\mathtt{s}}.
	\end{eqnarray}
	It follows from \eqref{case_1_equation_3}, \eqref{case_2_equation_1} and \eqref{case_2_equation_3} that
	$$ | I_{(\alpha,\beta)}^{[j]}(x,\xi) |  \leq  C_{\alpha\beta} \japxi^{1-m-\beta+p-j-\alpha} \japx^{1-\frac{p-k_1}{p-1}\sigma+\frac{p-k_2}{p-1}(1-\sigma)+\cdots+\frac{p-k_\mathtt{r}}{p-1}(1-\sigma)-\beta-\frac{p-j}{p-1}\sigma-\alpha+\mathtt{s}}.$$
	 
	Now we perform some computations on the exponent of $\japx$. Since $k_\mathtt{a} \geq 1$, for each $2 \leq \mathtt{a} \leq \mathtt{r}$, we obtain
	\begin{eqnarray}
		& & 1-\frac{p-k_1}{p-1}\sigma+\frac{p-k_2}{p-1}(1-\sigma)+\cdots+\frac{p-k_\mathtt{r}}{p-1}(1-\sigma)-\beta-\frac{p-j}{p-1}\sigma-\alpha+\mathtt{s} \nonumber \\
		& \leq & 1-\frac{p-k_1}{p-1}\sigma+\frac{p(\mathtt{r}-1)-(\mathtt{r}-1)}{p-1}(1-\sigma)-\alpha-\beta-\frac{p-j}{p-1}\sigma+\mathtt{s} \nonumber \\
		& = & 1-\frac{p-k_1}{p-1}\sigma+(\mathtt{r}-1)(1-\sigma)-\beta-\frac{p-j}{p-1}\sigma\underbrace{+\mathtt{s}-\alpha}_{\leq 0} \nonumber \\
		& \leq & 1-\frac{p-k_1}{p-1}\sigma+\underbrace{(\mathtt{r}-1)(1-\sigma)}_{\leq \mathtt{r}-1}-\beta-\frac{p-j}{p-1}\sigma \nonumber \\
		& \leq & \underbrace{1-\frac{p-k_1}{p-1}\sigma}_{\leq 1}+\mathtt{r}-1-\beta-\frac{p-j}{p-1}\sigma \nonumber \\
		& \leq & -\frac{p-j}{p-1}\sigma +1-1\underbrace{-\beta+\mathtt{r}}_{\leq 0} \leq -\frac{p-j}{p-1}\sigma. \nonumber
	\end{eqnarray}
	Hence
	$$
	| I_{(\alpha,\beta)}^{[j]}(x,\xi) | \leq C_{\alpha\beta} \japxi^{p-m-j+1-\alpha-\beta} \japx^{-\frac{p-j}{p-1}\sigma},
	$$
	which is the desired estimate.

	From \eqref{conjugation_of_a_(p-j)D^(p-j)} and the above discussion it follows that
	\beqsn
	e^\Lambda(x,D) \circ (i a_{p-j}(t,x)D_x^{p-j}) \circ \hskip2pt ^R (e^{-\Lambda}(x,D))& =& i a_{p-j}(t,x)D_x^{p-j} + \mathbf{op}\left(\sum_{m=2}^{p-j} c_{p-m}^{[j]}\right) 
\\
&+& r_0(t,x,D) + r_\infty(t,x,D),
	\eeqsn
	where $r_0$ has order zero, $r_\infty$ satisfies an estimate of the form \eqref{regularizingestimate} and $c_{p-m}^{[j]}$ has order $p-j-m+1$ with respect to $\xi$, $-\frac{p-j}{p-1}\sigma$ with respect to $x$ and depends only on $\lambda_{p-1}, ..., \lambda_{p-m+1}$. By composing with the Neumann series and using similar arguments as in the conjugation of the leading term, the previous structure of the composition does not change, which means that
		\beqs\label{lower_order_conjugation}
			e^\Lambda(x,D) \circ (i a_{p-j}(t,x)D_x^{p-j}) \circ (e^\Lambda(x,D))^{-1}& =& i a_{p-j}(t,x)D_x^{p-j} + \mathbf{op}\left(\sum_{m=2}^{p-1} d_{p-m}^{[j]}\right) 
\\
&+& r_0(t,x,D) + r_\infty(t,x,D),\nonumber
		\eeqs
	
	with $d_{p-m}^{[j]}$ of order $p-j-m+1$ with respect to $\xi$ and $-\frac{p-j}{p-1}\sigma$ with respect to $x$, depending only on $\lambda_{p-1}, ..., \lambda_{p-m+1}$, $r_0$ and $r_\infty$ new terms with the same properties as before.

	\item The conjugation of $ia_0(t,x)$: For this term, the conjugation is given by
	\begin{equation}\label{conjugation_of_a_(0)}
		e^\Lambda(x,D) \circ (i a_0(t,x)) \circ (e^\Lambda(x,D))^{-1} = r_0(t,x,D) + r_\infty(t,x,D), 
	\end{equation}
	where $r_0$ has order zero and $r_\infty$ satisfies an estimate of the form \eqref{regularizingestimate}.

\end{itemize}

Finally, let us gather all the previous computations. We obtain
\begin{eqnarray}\label{conjugation_of_iP_provisory}
	e^\Lambda (x,D)\circ (iP) \circ (e^\Lambda(x,D))^{-1} 
	& = & \partial_t + ia_p(t)D_x^p - \mathbf{op}\left(\partialxi a_p(t) \xi^p \partialx\Lambda\right) + \mathbf{op}\left(a_p(t) \sum_{m=2}^{p-1} \mathtt{c}_{p-m}\right) \nonumber \\
	& + & \sum_{j=1}^{p-1} ia_{p-j}(t,x)D_x^{p-j} + \mathbf{op}\left(\sum_{j=1}^{p-1}\sum_{m=2}^{p-j} d_{p-m}^{[j]}\right) + (r_0+r_\infty)(t,x,D).
\end{eqnarray}
Now we define
$$
d_{p-j}(t,x,\xi) := a_p(t)\mathtt{c}_{p-j}(x,\xi) + \sum_{\mathtt{h}+m=j+1} d_{p-m}^{[\mathtt{h}]}(x,\xi), \qquad j=2,...,p-1,
$$
which has $\xi$-order $p-j$, $x$-order $-\frac{p-j}{p-1}\sigma$ and depends only on $\lambda_{p-1},...,\lambda_{p-j+1}$. Hence we can rewrite \eqref{conjugation_of_iP_provisory} as
\begin{eqnarray}\label{conjugation_of_iP}
	e^\Lambda(x,D) \circ (iP) \circ (e^\Lambda(x,D))^{-1} & = & \partial_t + ia_p(t)D_x^p + \sum_{j=1}^{p-1} ia_{p-j}(t,x)D_x^{p-j} \nonumber \\
	& - & \mathbf{op}(\partialxi a_p(t) \xi^p \partialx \lambda_{p-1}) - \cdots - \mathbf{op}(\partialxi a_p(t) \xi^p \partialx \lambda_1) \nonumber \\
	& + & \mathbf{op}\left(\sum_{j=2}^{p-1} d_{p-j}\right) + (r_0+r_\infty)(t,x,D),
\end{eqnarray}
where $r_0$ is a zero order term, $r_\infty$ satisfies an estimate of the form \eqref{regularizingestimate} and $d_{p-j}$ has $\xi$-order $p-j$, $x$-order $-\frac{p-j}{p-1}\sigma$ and depends only on $\lambda_{p-1},...,\lambda_{p-j+1}$.

In order to perform the next conjugation, let us rewrite \eqref{conjugation_of_iP} in the following way: 
\begin{eqnarray}\label{conjugation_of_iP_improved}
	e^\Lambda(x,D) \circ (iP) \circ (e^\Lambda(x,D))^{-1} & = & \partial_t + a_p(t)D_x^p \nonumber \\
	& + & ia_{p-1}(t,x)D_x^{p-1} - \mathbf{op}(\partialxi a_p(t)\xi^p \partialx\lambda_{p-1}) \nonumber \\
	& + & ia_{p-2}(t,x)D_x^{p-2} - \mathbf{op}(\partialxi a_p(t)\xi^p \partialx\lambda_{p-2}) + \mathbf{op}(d_{p-2}) \nonumber \\
	& + & \cdots \nonumber \\
	& + & ia_2(t,x)D_x^2 - \mathbf{op}(\partialxi a_p(t)\xi^p \partialx\lambda_2) + \mathbf{op}(d_2) \nonumber \\
	& + & ia_1(t,x)D_x - \mathbf{op}(\partialxi a_p(t)\xi^p \partialx\lambda_1) + \mathbf{op}(d_1) \nonumber \\
	& + & (r_0 + r_\infty)(t,x,D),
\end{eqnarray}
where $d_{p-j}$, $j=2,...,p-1$, $r_0$ and $r_\infty$ are as described previously.

\subsection{Conjugation of $e^\Lambda(iP)(e^\Lambda)^{-1}$ by $e^{\Lambda_{K,\rho^\prime}}$}

For some $K > 0$ and $\rho^\prime < \rho$, we consider the operator $e^{\Lambda_{K,\rho^\prime}}(t,D)$, where
$$
\Lambda_{K,\rho^\prime}(t,\xi) := K(T-t)\japxih^{(p-1)(1-\sigma)} + \rho^\prime\japxih^{1/\theta}.
$$
is a symbol of order $1/\theta$, since $(p-1)(1-\sigma) < 1/\theta$ by assumption.\\

In order to apply again Theorem \ref{theorem_conjugation} we observe that
$$|\Lambda_{K,\rho^\prime}(t,\xi)|\leq \left(K Th^{(p-1)(1-\sigma)-1/\theta} + \rho^\prime\right)\japxih^{1/\theta}$$
fulfills \eqref{firstasslambda} (for $h$ large enough) for every $\rho_0>\rho'$. So, 
enlarging $h_0$ such that $h \geq h_0$ and taking $\rho'$ small enough, Theorem \ref{theorem_conjugation} applies. 
\begin{itemize}
	
	\item {\bf Conjugation of $\partial_t$:} for this term we get
	$$
	e^{\Lambda_{K,\rho^\prime}}(t,D) \circ \partial_t \circ e^{-\Lambda_{K,\rho^\prime}}(t,D) = \partial_t + K\langle D_x \rangle_h^{(p-1)(1-\sigma)}.
	$$

	\item {\bf Conjugation of $ia_p(t)D_x^p$:} since $a_p(t)$ does not depend on $x$, the conjugation is simply given by
	$$
	e^{\Lambda_{K,\rho^\prime}}(t,D) \circ (ia_{p}(t)D_x^p) \circ e^{-\Lambda_{K,\rho^\prime}}(t,D) = ia_p(t)D_x^p.
	$$

	\item {\bf Conjugation of $ia_{p-1}(t,x)D_x^{p-1} - \mathbf{op}(\partialxi a_p(t) \xi^p \partialx\lambda_{p-1})$:} we have that
	\begin{eqnarray}
		& & e^{\Lambda_{K,\rho^\prime}}(t,D) \circ \left(ia_{p-1}(t,x)D_x^{p-1} - \mathbf{op}(\partialxi a_p(t) \xi^p \partialx\lambda_{p-1})\right) \circ e^{-\Lambda_{K,\rho^\prime}}(t,D) \nonumber \\
		& = & ia_{p-1}(t,x)D_x^{p-1} - \mathbf{op}(\partialxi a_p(t) \xi^p \partialx\lambda_{p-1}) + A_{K,\rho^\prime}^{[p-1]}(t,x,D), \nonumber
	\end{eqnarray}
	with $A_{K,\rho^\prime}^{[p-1]}(t,x,D)$ a remainder term whose symbol satisfies 
	\begin{equation}\label{remainder_p-1}
		\left| \doublepartial A_{K,\rho^\prime}^{[p-1]}(t,x,\xi) \right| \leq C_{T,K,\rho^\prime,M_1}\japxih^{p-1-\left(1-\frac1{\theta}\right)} \japx^{-\sigma-1},
	\end{equation}
	for all $\alpha,\beta \in \Bbb{N}_0$, $t \in [0,T]$ and $x,\xi \in \Bbb{R}$.

	\item {\bf Conjugation of $ia_{p-j}(t,x)D_x^{p-j} - \mathbf{op}(\partialxi a_p(t) \xi^p \partialx\lambda_{p-j}) + \mathbf{op}(d_{p-j})$, for $j=2,...,p-1$:} for every $j$, we have
	\begin{eqnarray}
		& & e^{\Lambda_{K,\rho^\prime}}(t,D) \circ \left(ia_{p-j}(t,x)D_x^{p-j} - \mathbf{op}(\partialxi a_p(t) \xi^p \partialx\lambda_{p-j}) + \mathbf{op}(d_{p-j})\right) \circ e^{-\Lambda_{K,\rho^\prime}}(t,D) \nonumber \\
		& = & ia_{p-j}(t,x)D_x^{p-j} - \mathbf{op}(\partialxi a_p(t) \xi^p \partialx\lambda_{p-j}) + \mathbf{op}(d_{p-j}) + A_{K,\rho^\prime}^{[p-j]}(t,x,D), \nonumber
	\end{eqnarray}
	with $A_{K,\rho^\prime}^{[p-j]}(t,x,D)$ a remainder term whose symbol satisfies
	\begin{equation}\label{remander_p-j}
		\left| \doublepartial A_{K,\rho^\prime}^{[p-j]}(t,x,\xi) \right| \leq C_{T,K,\rho^\prime,M_{p-1},...,M_{p-j}} \japxih^{p-j-\left(1-\frac{1}{\theta}\right)} \japx^{-\frac{p-j}{p-1}\sigma},
	\end{equation}
	for all $\alpha,\beta \in \Bbb{N}_0$, $t \in [0,T]$ and $x,\xi \in \R$. 
	
\end{itemize}

Gathering all these computations, we get
\begin{eqnarray}\label{conjugation_of_e_iP_e}
	 Q_{\Lambda,K,\rho^\prime}(t,x,D) &\circ& (iP) \circ (Q_{\Lambda,K,\rho^\prime}(t,x,D))^{-1} \nonumber \\ 
	& = & e^{\Lambda_{K,\rho^\prime}}(t,D) \circ e^\Lambda(x,D) \circ (iP) \circ (e^\Lambda(x,D))^{-1} \circ e^{-\Lambda_{K,\rho^\prime}}(t,D)  \nonumber
\\
&=& \partial_t + a_p(t)D_x^p 
	 +  ia_{p-1}(t,x)D_x^{p-1} - \mathbf{op}(\partialxi a_p(t) \xi^p \partialx\lambda_{p-1}) + A_{K,\rho^\prime}^{[p-1]}(t,x,D) \nonumber \\
	& + &  ia_{p-2}(t,x)D_x^{p-2} - \mathbf{op}(\partialxi a_p(t) \xi^p \partialx\lambda_{p-2}) + \mathbf{op}(d_{p-2}) + A_{K,\rho^\prime}^{[p-2]}(t,x,D) \nonumber \\
	& + & \cdots \nonumber \\
	& + & ia_2(t,x)D_x^2 - \mathbf{op}(\partialxi a_p(t) \xi^p \partialx\lambda_2) + \mathbf{op}(d_2) + A_{K,\rho^\prime}^{[2]}(t,x,D) \nonumber \\
	& + & ia_1(t,x)D_x - \mathbf{op}(\partialxi a_p(t) \xi^p \partialx\lambda_1) + \mathbf{op}(d_1) + A_{K,\rho^\prime}^{[1]}(t,x,D) \nonumber \\
	& + & K \langle D_x \rangle_h^{(p-1)(1-\sigma)} + (r_0+r_\infty)(t,x,D),
\end{eqnarray}
where $A_{K,\rho^\prime}^{[p-j]}$ is a remainder depending on $M_{p-1},...,M_{p-j}$ satisfying \eqref{remainder_p-1} for $j=1$ and \eqref{remander_p-j} for $j=1,...,p-1$.

\section{Estimates from below for the real parts}\label{estimatesfrombelow}

This section is devoted to prove some estimates from below for the real parts of the lower order terms of
$$
(iP)_{\Lambda,K,\rho^\prime} := Q_{\Lambda,K,\rho^\prime} \circ iP \circ (Q_{\Lambda,K,\rho^\prime})^{-1},
$$
with the aim to apply to these terms the sharp G{\aa}rding inequality and finally to achieve a well-posedness result for the Cauchy problem \eqref{auxiliary_cauchy_problem_gevrey}.  To this purpose, let us rewrite \eqref{conjugation_of_e_iP_e} in a different way.
For each $j=1,...,p-1$, by the definition of $\lambda_{p-j}$ it follows that
$$
-\partialxi\xi^p\partialx\lambda_{p-j}(x,\xi) = -p\xi^{p-1} M_{p-j} \omega\left(\frac{\xi}{h}\right) \japxih^{1-j} \japx^{-\frac{p-j}{p-1}\sigma} \psi\left(\frac{\japx}{\japxih^{p-1}}\right).
$$
By definition \eqref{z} of $\omega$, for $|\xi|>2h$ we can rewrite the above expression as
\begin{eqnarray}
	-\partialxi\xi^p\partialx\lambda_{p-j}(x,\xi) & = & p|\xi|^{p-1} M_{p-j} \japxih^{1-j} \japx^{-\frac{p-j}{p-1}\sigma} \psi\left(\frac{\japx}{\japxih^{p-1}}\right) \nonumber \\
	& = & p|\xi|^{p-1} M_{p-j} \japxih^{1-j} \japx^{-\frac{p-j}{p-1}\sigma} \nonumber \\
	& - & p|\xi|^{p-1} M_{p-j} \japxih^{1-j} \japx^{-\frac{p-j}{p-1}\sigma} \left[1 - \psi\left(\frac{\japx}{\japxih^{p-1}}\right)\right]. \nonumber 
\end{eqnarray}
The term $p|\xi|^{p-1} M_{p-j} \japxih^{1-j} \japx^{-\frac{p-j}{p-1}\sigma}$ has order $p-j$ in $\xi$ and $-\frac{p-j}{p-1}\sigma$ in $x$. For the term $p|\xi|^{p-1} M_{p-j} \japxih^{1-j} \japx^{-\frac{p-j}{p-1}\sigma} \left[1 - \psi\left(\frac{\japx}{\japxih^{p-1}}\right)\right]$, let us make some considerations: first of all, by definition \eqref{z} of $\psi$, we see  that $1-\psi\left(\japx/\japxih^{p-1}\right)$ is supported for $2\japx\geq\japxih^{p-1}$. From this, for each $j=1,...,p-1$, we have
$$
\japx^{-\frac{p-j}{p-1}\sigma} \leq 2 \japxih^{-(p-j)\sigma},
$$
hence
$$
\japxih^{1-j}\japx^{-\frac{p-j}{p-1}\sigma} \leq 2\japxih^{1-j-(p-j)\sigma}.
$$
Therefore, $p|\xi|^{p-1} M_{p-j} \japxih^{1-j} \japx^{-\frac{p-j}{p-1}\sigma} \left[1 - \psi\left(\frac{\japx}{\japxih^{p-1}}\right)\right]$ can be considered of order less than or equal to $(p-1)(1-\sigma)$ in $\xi$ and order zero in $x$, for each $j=1,...,p-1$, and then the sum of these terms still has order less than or equal to $(p-1)(1-\sigma)$ in $\xi$ and zero in $x$. It follows that
$$
\sum_{j=1}^{p-1} \mathbf{op}(-\partialxi a_p(t) \xi^p \partialx \lambda_{p-j}(x,\xi)) = \sum_{j=1}^{p-1} p M_{p-j} a_p(t) |\xi|^{p-1} \japxih^{1-j} \japx^{-\frac{p-j}{p-1}\sigma} - B^{[(p-1)(1-\sigma)]}(t,x,D), 
$$
where 
\begin{equation}\label{Bp-1}B^{[(p-1)(1-\sigma)]}(t,x,\xi) := \sum_{j=1}^{p-1} pa_p(t)|\xi|^{p-1} M_{p-j} \japxih^{1-j} \japx^{-\frac{p-j}{p-1}\sigma} \left[1 - \psi\left(\frac{\japx}{\japxih^{p-1}}\right)\right],
\end{equation} that is, $B^{[(p-1)(1-\sigma)]}$ has order less than or equal to $(p-1)(1-\sigma)$ in $\xi$, order zero in $x$ and it depends on $M_{p-1}, ..., M_1$. 

With all these considerations we can write \eqref{conjugation_of_e_iP_e} as
\begin{eqnarray}\label{conjugation_of_e_iP_e_improved}
	(iP)_{\Lambda,K,\rho^\prime}&=& \partial_t + ia_p(t)D_x^p \nonumber \\
	& + & ia_{p-1}(t,x)D_x^{p-1} + \mathbf{op}\left(p M_{p-1} a_p(t) |\xi|^{p-1} \japx^{-\sigma}\right) +A_{K,\rho^\prime}^{[p-1]}(t,x,D) \nonumber \\
	& + & ia_{p-2}(t,x)D_x^{p-2} + \mathbf{op}\left(p M_{p-2} a_p(t) |\xi|^{p-1} \japxih^{-1} \japx^{-\frac{p-2}{p-1}\sigma}\right) + \mathbf{op}(d_{p-2}) + A_{K,\rho^\prime}^{[p-2]}(t,x,D) \nonumber \\
	& + & \cdots \nonumber \\
	& + & ia_2(t,x)D_x^2 + \mathbf{op}\left(p M_2 a_p(t) |\xi|^{p-1} \japxih^{-p+3} \japx^{-\frac{2}{p-1}\sigma}\right) + \mathbf{op}(d_2) + A_{K,\rho^\prime}^{[2]}(t,x,D) \nonumber\\
	& + & ia_1(t,x)D_x + \mathbf{op}\left(p M_1 a_p(t)|\xi|^{p-1} \japxih^{-p+2} \japx^{-\frac{1}{p-1}\sigma}\right) + \mathbf{op}(d_1) + A_{K,\rho^\prime}^{[1]}(t,x,D) \nonumber \\
	& + & K \langle D_x \rangle_h^{(p-1)(1-\sigma)} -  B^{[(p-1)(1-\sigma)]}(t,x,D)  + (r_0 + r_\infty)(t,x,D),
\end{eqnarray}
with $d_{p-m}$, $m=2,...,p-1$, $A_{K,\rho^\prime}^{[p-j]}$, $j=1,...,p-1$, $r_0$, $r_\infty$ and $B^{[(p-1)(1-\sigma)]}$ as described previously.
We are now ready for the desired estimates from below.

\begin{itemize}
	
	\item Estimates for level $p-1$: let
	$$
	\tilde{a}_{p-1}(t,x,D_x) := ia_{p-1}(t,x)D_x^{p-1} + \mathbf{op}\left(p M_{p-1} a_p(t) |\xi|^{p-1} \japx^{-\sigma}\right) +A_{K,\rho^\prime}^{[p-1]}(t,x,D).
	$$
	We have that
	\begin{equation}\label{level_p-1_provisory}
		\mathbf{Re} \, \tilde{a}_{p-1}(t,x,\xi) = - \mathbf{Im} \, a_{p-1}(t,x)\xi^{p-1} + p M_{p-1} a_p(t) |\xi|^{p-1} \japx^{-\sigma} + \mathbf{Re}\,  A_{K,\rho^\prime}^{[p-1]}(t,x,\xi).
	\end{equation}
	By the assumption (ii) in Theorem \ref{maintheorem1} we obtain
	$$
	|\mathbf{Im} \, a_{p-1}(t,x)\xi^{p-1}| \leq C_{a_{p-1}} \japx^{-\sigma} |\xi|^{p-1},
	$$
	and from \eqref{remainder_p-1},
	\begin{eqnarray}
		|\mathbf{Re} \, A_{K,\rho^\prime}^{[p-1]}(t,x,\xi)| & \leq & C_{T,K,\rho^\prime,M_{p-1}} \japxih^{p-1-1+\frac{1}{\theta}} \japx^{-\sigma} \nonumber \\
		& \leq & C_{T,K,\rho^\prime,M_{p-1}} \japxih^{p-1} h^{-1+\frac{1}{\theta}} \japx^{-\sigma}. \nonumber
	\end{eqnarray}
	Moreover, $|\xi| \geq 2h$ implies that $\japxih \leq \frac{\sqrt{5}}{2}|\xi|$. Therefore we can estimate \eqref{level_p-1_provisory} by
	\begin{equation}\label{level_p-1}
		\mathbf{Re} \, \tilde{a}_{p-1}(t,x,\xi) \geq \left(-C_{a_{p-1}} + pM_{p-1}C_{a_p} - C_{T,K,\rho^\prime,M_{p-1}}h^{-1+\frac{1}{\theta}}\right) \japx^{-\sigma} |\xi|^{p-1}.
	\end{equation}
	Now we can choose the real number $M_{p-1}$ such that $-C_{a_{p-1}} + pM_{p-1}C_{a_p} \geq 1$, which is equivalent to
	$$\boxed{
	M_{p-1} \geq \frac{1+C_{a_{p-1}}}{pC_{a_p}}.}
	$$
	If  we choose $h$ large enough (at the moment, we need only 
	$C_{T,K,\rho^\prime,M_{p-1}} h^{-1+\frac{1}{\theta}} \leq \frac{1}{2}$)
	we get
	$$
	\mathbf{Re} \, \tilde{a}_{p-1}(t,x,\xi) \geq \frac{1}{2} \japx^{-\sigma} |\xi|^{p-1}, \quad |\xi| \geq 2h.
	$$
	By Theorem 6 in \cite{ACsharpGarding}, the operator $\tilde{a}_{p-1}(t,x,D)$ can be decomposed as
	\begin{equation}\label{sharp_garding_level_p-1}
		\tilde{a}_{p-1}(t,x,D) = Q_+^{[p-1]}(t,x,D) + \tilde{r}^{[p-2]}(t,x,D),
	\end{equation}
	where $$\mathbf{Re} \, \left(Q_+^{[p-1]}(t,x,D)u,u\right)_{L^2} \geq 0, \quad u \in L^2(\R),$$ and $\tilde{r}^{[p-2]}(t,x,\xi) \in \mathbf{SG}_{\mu^\prime}^{p-2,-1-\sigma}(\R^2)$, $t \in [0,T]$, depends only on $M_{p-1}$ and $h$. Thus, we can put $\tilde{r}^{[p-2]}(t,x,D)$ together with all the other terms of order $p-2$ in \eqref{conjugation_of_e_iP_e_improved}, since it presents the same order in $\xi$ and stronger decay as $|x|\to\infty$.
	
	\item Estimates for level $p-2$: Let
	\begin{eqnarray}
		\tilde{a}_{p-2}(t,x,D) & := & ia_{p-2}(t,x)D_x^{p-2} + \mathbf{op}\left(pM_{p-2}a_p(t)|\xi|^{p-1}\japxih^{-1}\japx^{-\frac{p-2}{p-1}\sigma}\right) \nonumber \\
		& + & A_{K,\rho^\prime}^{[p-2]}(t,x,D) + d_{p-2}(t,x,D) + \tilde{r}^{[p-2]}(t,x,D), \nonumber
	\end{eqnarray}
	hence
	\begin{eqnarray}\label{level_p-2_provisory}
		\mathbf{Re} \, \tilde{a}_{p-2}(t,x,\xi) & = & -\mathbf{Im} \, a_{p-2}(t,x)\xi^{p-2} + pM_{p-2}a_p(t)|\xi|^{p-1}\japxih^{-1}\japx^{-\frac{p-2}{p-1}\sigma} \nonumber \\
		& + & \mathbf{Re} \, A_{K,\rho^\prime}^{[p-2]}(t,x,\xi) + \mathbf{Re} \, d_{p-2}(t,x,\xi) + \tilde{r}^{[p-2]}(t,x,\xi).
	\end{eqnarray}
	Since $d_{p-2}$ and $\tilde{r}^{[p-2]}$ depend only on $M_{p-1}$ and satisfy
	$$
	|\mathbf{Re} \, d_{p-2}(t,x,\xi) + \mathbf{Re} \, \tilde{r}^{[p-2]}(t,x,\xi)| \leq C_{M_{p-1}} \japxih^{p-2} \japx^{-\frac{p-2}{p-1}\sigma},
	$$
	by similar arguments and estimates to the ones we have used in level $p-1$ and with the help of \eqref{remander_p-j},  we estimate \eqref{level_p-2_provisory} as
	\begin{eqnarray}\label{level_p-2}
		& & \mathbf{Re} \, \tilde{a}_{p-2}(t,x,\xi) \nonumber \\
		& \geq & \left(-C_{a_{p-2}} + pM_{p-2}C_{a_p} - C_{M_{p-1}} - C_{T,K,\rho^\prime,M_{p-1},M_{p-2}} h^{-1+\frac{1}{\theta}}\right) |\xi|^{p-2} \japx^{-\frac{p-2}{p-1}\sigma}.
	\end{eqnarray}
	Now, choosing $M_{p-2}$ such that $-C_{a_{p-2}} + pM_{p-2}C_{a_p} - C_{M_{p-1}} \geq 1$, we obtain
	$$\boxed{
	M_{p-2} \geq \frac{1+C_{a_{p-2}}+C_{M_{p-1}}}{pC_{a_p}}}
	$$
	depending only on $M_{p-1}$. Again, if we choose $h$ large enough such that 
	$
	C_{T,K,\rho^\prime,M_{p-1},M_{p-2}} h^{-1+\frac{1}{\theta}} \leq \frac{1}{2}
	$,
	we get
	$$
	\mathbf{Re} \, \tilde{a}_{p-2}(t,x,\xi) \geq \frac{1}{2} \japx^{-\frac{p-2}{p-1}\sigma} |\xi|^{p-2}, \quad |\xi| \geq 2h.
	$$
	Again by Theorem 6 in \cite{ACsharpGarding}, it follows that
	\begin{equation}\label{sharp_garding_level_p-2}
		\tilde{a}_{p-2} = Q_+^{[p-2]}(t,x,D) + \tilde{r}^{[p-3]}(t,x,D),
	\end{equation}
	with $$\mathbf{Re}\, \left(Q_+^{[p-2]}(t,x,D)u,u\right)_{L^2} \geq 0,\quad u \in L^2(\R),$$ and $\tilde{r}^{[p-3]}(t,x,\xi) \in \mathbf{SG}_{\mu^\prime}^{p-3,-\frac{p-2}{p-1}\sigma-1}(\R^2)$, $t \in [0,T]$, depends only on $M_{p-1},M_{p-2}$, and $h$, so it can be put together with all the other terms of order $p-3$.
	
	\item For a generic level $(p-j)$, $j\geq2$: we call
	\begin{eqnarray}
		\tilde{a}_{p-j}(t,x,D) & := & ia_{p-j}(t,x)D_x^{p-j} + \mathbf{op}\left(pM_{p-j}a_p(t)|\xi|^{p-1}\japxih^{1-j}\japx^{-\frac{p-j}{p-1}\sigma}\right) \nonumber \\
		& + &   A_{K,\rho^\prime}^{[p-j]}(t,x,D) + d_{p-j}(t,x,D) + \tilde{r}^{[p-j]}(t,x,D). \nonumber
	\end{eqnarray}
	By similar arguments as in the previous steps, we get
	\begin{eqnarray}\label{level_p-j}
		&&\mathbf{Re} \, \tilde{a}_{p-j}(t,x,\xi)  \geq \nonumber \\
		& & \left(-C_{a_{p-j}} + pM_{p-j}C_{a_p} - C_{M_{p-1},...,M_{p-j+1}} - C_{T,K,\rho^\prime,M_{p-1},...,M_{p-j}}h^{-1+\frac{1}{\theta}}\right) |\xi|^{p-j}\japx^{-\frac{p-j}{p-1}\sigma}.
	\end{eqnarray}
	By choosing $M_{p-j}$ such that $-C_{a_{p-j}} + pM_{p-j}C_{a_p} - C_{M_{p-1},...,M_{p-j+1}} \geq 1$, we obtain
	$$
	\boxed{M_{p-j} \geq \frac{1+C_{a_{p-j}}+C_{M_{p-1},...,M_{p-j+1}}}{pC_{a_p}}}
	$$
	depending on $M_{p-1},...,M_{p-j+1}$, and choosing $h$ large such that 
	$C_{T,K,\rho^\prime,M_{p-1},...,M_{p-j}}h^{-1+\frac{1}{\theta}} \leq \frac{1}{2}
	$, we have
	$$
	\mathbf{Re} \, \tilde{a}_{p-j}(t,x,\xi) \geq \frac{1}{2} \japx^{-\frac{p-j}{p-1}\sigma} |\xi|^{p-j}, \quad |\xi| \geq 2h.
	$$
	By Theorem 6 in \cite{ACsharpGarding}
	\begin{equation}\label{sharp_garding_level_p-j}
		\tilde{a}_{p-j}(t,x,D) = Q_+^{[p-j]}(t,x,D) + \tilde{r}^{[p-j-1]}(t,x,D),
	\end{equation}
	where $$\mathbf{Re}\,\left(Q_+^{[p-j]}(t,x,D)u,u\right)_{L^2} \geq 0, \quad u \in L^2(\R)$$ and $r^{[p-j-1]}(t,x,\xi) \in \mathbf{SG}_{\mu^\prime}^{p-j-1,-1-\frac{p-j}{p-1}\sigma}(\R^2)$ depends on $M_{p-1},...,M_{p-j+1},h$.
	
\end{itemize}

At this point, after the choice of all the parameters $M_{p-1},...,M_1$, we get that if we choose $h$ large enough (at the moment we need $C_{T,K,\rho^\prime,M_{p-1},...,M_{1}}h^{-1+\frac{1}{\theta}} \leq \frac{1}{2}$) the conjugated operator can be expressed as
\begin{eqnarray}
	 (iP)_{\Lambda,K,\rho^\prime} &=& \partial_t + ia_p(t)D_x^p \nonumber \\
	& + & Q_+^{[p-1]}(t,x,D) + Q_+^{[p-2]}(t,x,D) + \cdots + Q_+^{[2]}(t,x,D)+ Q_+^{[1]}(t,x,D) \nonumber \\
	& + & K \langle D_x \rangle_h^{(p-1)(1-\sigma)} - B^{[(p-1)(1-\sigma)]}(t,x,D) + (r_0 + r_\infty)(t,x,D). \nonumber
\end{eqnarray}

Now, recalling the definition \eqref{Bp-1} of $B^{[(p-1)(1-\sigma)]}$ we can choose $$\boxed{K \geq K_0(M_{p-1},...,M_1)}$$ large enough such that
$$
\mathbf{Re} \, \left(K\japxih^{(p-1)(1-\sigma)} - B^{[(p-1)(1-\sigma)]}(t,x,\xi)\right) \geq 0,
$$
and by applying Theorem 6 in \cite{ACsharpGarding} once more we obtain
\begin{eqnarray}\label{sharp_garding_last_step}
  (iP)_{\Lambda,K,\rho^\prime}&= &   \partial_t + ia_p(t)D_x^p +\sum_{j=1}^{p-1} Q_+^{[p-j]}(t,x,D) \nonumber \\
	& + &  Q_+^{[(p-1)(1-\sigma)]}(t,x,D) + (r_0 + r_\infty)(t,x,D).
\end{eqnarray}
Summing up, for $\rho'$ small enough, choosing, in the order, $M_{p-1},...,M_1,K$, and $h$ large enough, precisely  $$\boxed{C_{T,K,\rho^\prime,M_{p-1},...,M_{1}}h^{-1+\frac{1}{\theta}} \leq \frac{1}{2}\quad{\rm and}\quad h>h_0,}$$
enlarging the parameter $h_0 > 0$ given in Proposition \ref{proposition_inverse_of_e^lambda} if necessary, we obtain that formula \eqref{sharp_garding_last_step} holds.

By these estimates, we can assert the next proposition.

\begin{proposition}\label{proposition_auxiliar_energy}
	Assume $\rho' >0$ sufficiently small. Then there exist $M_{p-1},...,M_1>0$, $K>0$ and $h_0=h_0(K,M_{p-1},...,M_1,T,\rho^\prime)>0$ such that for every $h>h_0$ the Cauchy problem associated to the conjugated operator \eqref{sharp_garding_last_step} is well-posed in $H^m(\R)$. More precisely, for any Cauchy data $\tilde{f}\in C([0,T];H^m(\R))$ and $\tilde{g}\in H^m(\R)$, there exists a unique solution $v\in C([0,T];H^m(\R))\cap C^1([0,T];H^{m-p}(\R))$ such that the following energy estimates holds
	$$
	\|v(t)\|_{H^m}^2 \leq C \left( \|\tilde{g}\|_{H^m}^2 + \int_0^t \|\tilde{f}(\tau)\|_{H^m}^2 d\tau \right), \quad t \in [0,T].
	$$
\end{proposition}

\begin{proof}
	First of all, let us consider the choices of the constants which make formula \eqref{sharp_garding_last_step} true. Then it follows that
	\begin{eqnarray}\label{before_gronwall_1}
		\frac{d}{dt} \|v(t)\|_{L^2}^2 & = & 2 \mathbf{Re} \, \langle P_{\Lambda,K,\rho^\prime}(t,x,D) v(t), v(t)\rangle_{L^2} - 2 \mathbf{Re} \, \langle i a_p(t)D_x^p v(t), v(t)\rangle_{L^2} \nonumber \\
		& - & 2 \mathbf{Re} \, \left\langle \sum_{j=1}^{p-1} Q_+^{[p-j]}(t,x,D) v(t), v(t) \right\rangle_{L^2} - 2 \mathbf{Re} \, \left\langle Q_+^{[(p-1)(1-\sigma)]}(t,x,D) v(t), v(t) \right\rangle_{L^2} \nonumber \\
		& - & 2 \mathbf{Re} \, \left\langle r_0(t,x,D) v(t), v(t)\right\rangle_{L^2} \nonumber \\
		& \leq & 2 \|P_{\Lambda,K,\rho^\prime} v(t)\|_{L^2} \|v(t)\|_{L^2} + 2 \|r_0 v(t)\|_{L^2} \|v(t)\|_{L^2},
	\end{eqnarray}
	where, in order to obtain this inequality, we have used the Cauchy-Schwarz inequality in the first and last terms and that 
	$$
	\mathbf{Re}\left\langle ia_p(t)D_x^pv(t),v(t)\right\rangle_{L^2}=0, \quad \mathbf{Re}\left\langle Q_+^{p-j}v(t),v(t)\right\rangle_{L^2}\geq0, \quad \mathbf{Re}\left\langle Q_+^{[(p-1)(1-\sigma)]}v(t),v(t)\right\rangle_{L^2}\geq0.
	$$
	Since  $r_0$ has order zero, we still can estimate in \eqref{before_gronwall_1} as follows
	$$
	\frac{d}{dt}\|v(t)\|_{L^2}^2 \leq C^\prime \left(  \| v(t) \|_{L^2}^2 + \|P_{\Lambda,K,\rho^\prime} v(t)\|_{L^2}^2\right).
	$$
	Finally, by using Gronwall inequality and taking to account that $e^{C^\prime t}\leq e^{C^\prime T}=:C$, for $t\in[0,T]$, we get the energy estimate
	$$
	\|v(t)\|_{L^2}^2 \leq C \left(\|v(0)\|_{L^2}^2 + \int_0^t \|P_{\Lambda,K,\rho^\prime}v(\tau)\|_{L^2}^2 d\tau\right),
	$$
	which implies the well-posedness of \eqref{auxiliary_cauchy_problem_gevrey} in $H^m(\R)$.
\end{proof}

\section{Proof of Theorem \ref{maintheorem1}}\label{proofmainthm}

Given $\theta>1$ satisfying the hypothesis of Theorem \ref{maintheorem1}, $m\in\R$, $\rho>0$, and initial data
$$
f \in C\left([0,T];H_{\rho;\theta}^m(\R)\right), \quad g \in H_{\rho;\theta}^m(\R).
$$
Consider $\rho' <\rho$ and the positive constants $M_{p-1},...,M_1,K,h_0$ for which the Proposition \ref{proposition_auxiliar_energy} holds. We know that both symbols $\Lambda$ and $K(T-t)\langle\cdot\rangle_h^{(p-1)(1-\sigma)}$ have order $(p-1)(1-\sigma)<1/\theta$, hence by Proposition \ref{contgev} it follows that
\begin{eqnarray}
	f_{\Lambda,K,\rho^\prime} := Q_{\Lambda,K,\rho^\prime}(t,x,D)f & \in &  C\left([0,T];H^m(\R)\right) \nonumber \\
	g_{\Lambda,K,\rho^\prime} := Q_{\Lambda,K,\rho^\prime}(t,x,D)g & \in & H^m(\R) \nonumber,
\end{eqnarray}
for $\rho^\prime<\rho$. By Proposition \ref{proposition_auxiliar_energy}, there exists a unique solution $v\in C\left([0,T];H^m(\R)\right)$ to the Cauchy problem
$$
\left\lbrace \begin{array}{l}
	P_{\Lambda,K,\rho^\prime}v(t,x) = f_{\Lambda,K,\rho^\prime}(t,x) \\ 
	v(0,x) = g_{\Lambda,K,\rho^\prime}(x)
\end{array}, \right. \quad (t,x) \in [0,T] \times \R,
$$
satisfying the energy estimate
\begin{equation}\label{auxiliar_energy_estimate}
	\|v(t)\|_{H^m}^2 \leq C \left(\|g_{\Lambda,K,\rho^\prime}\|_{H^m}^2 + \int_0^t \|f_{\Lambda,K,\rho^\prime}(\tau)\|_{H^m}^2 d\tau\right), \quad t \in [0,T].
\end{equation}

By setting $u:=\left(Q_{\Lambda,K,\rho^\prime}(t,x,D)\right)^{-1}v$, we obtain a solution to the original Cauchy problem \eqref{cauchy_problem_gevrey}.\\
The next step is to figure out which space the solution $u$ belongs to. Notice that
by the definitions of $\Lambda_{K,\rho^\prime}$ and $\Lambda$, and by Lemma \ref{lemma_4_of_AA22} we may write
$$
u(t,x) = \hskip2pt ^R \left(e^{-\Lambda}(x,D)\right) \sum_j (-r(x,D))^j e^{-K(T-t)\langle D \rangle_h^{(p-1)(1-\sigma)}} e^{-\rho^\prime \langle D \rangle_h^\frac{1}{\theta}} v(t,x), \quad v \in H^m(\R),
$$
but $v \in H^m(\R)$ implies that $e^{-\rho^\prime\langle D \rangle_h^\frac{1}{\theta}}v=:u_1 \in H_{\rho^\prime;\theta}^m(\R)$, hence
$$
u(t,x) = \hskip2pt ^R \left(e^{-\Lambda}(x,D)\right) \sum_j (-r(x,D))^j e^{-K(T-t)\langle D \rangle_h^{(p-1)(1-\sigma)}}u_1, \quad u_1 \in H_{\rho^\prime;\theta}^m(\R).
$$
Notice that, for every $\delta_1>0$, we have
$$
e^{-K(T-t)\langle D \rangle_h^{(p-1)(1-\sigma)}}u_1 = \underbrace{e^{-K(T-t)\langle D \rangle_h^{(p-1)(1-\sigma)}} e^{-\delta_1 \langle D \rangle_h^\frac{1}{\theta}}}_{\text{order zero}} e^{\delta_1 \langle D \rangle_h^\frac{1}{\theta}} u_1 =: u_2 \in H_{\rho^\prime-\delta_1;\theta}^m(\R),
$$
and since $\sum_j(-r(x,D))^j$ has order zero, $\sum_j(-r(x,D))^j u_2 =: u_3 \in H_{\rho^\prime-\delta_1;\theta}^m(\R)$, which allows us to write
$$
u(t,x) =  \hskip2pt ^R \left(e^{-\Lambda}(x,D)\right) \sum_j (-r(x,D))^j u_2 = \hskip2pt ^R \left(e^{-\Lambda}(x,D)\right) u_3, \quad u_3 \in H_{\rho^\prime-\delta_1;\theta}^m(\R).
$$
By Proposition \ref{contgev}, we have that $\hskip2pt ^R \left(e^{-\Lambda}(x,D)\right)$ maps $H_{\rho;\theta}^m(\R)$ into $H_{\rho-\delta_2;\theta}^m(\R)$, for any $\delta_2>0$, and we can assert that $u(t,\cdot)\in H_{\rho^\prime -\delta;\theta}^m(\R)$ for all $\delta>0$, $t\in[0,T]$. Notice that the solution is less regular than the Cauchy data, in the sense that this solution exhibits a loss in the coefficient of the exponential weight. If we set $\tilde{\rho}:=\rho^\prime-\delta$, it follows from \eqref{auxiliar_energy_estimate} that
\begin{eqnarray}
	\|u(t)\|_{H_{\tilde{\rho};\theta}^m}^2 & = & \|\left(Q_{\Lambda,K,\rho^\prime}(t,\cdot,D)\right)^{-1} v(t)\|_{H_{\tilde{\rho};\theta}^m}^2 \leq C_1 \|v(t)\|_{H^m}^2 \nonumber \\
	& \leq & C_2 \left(\|g_{\Lambda,K,\rho^\prime}\|_{H^m}^2 + \int_0^t \|f_{\Lambda,K,\rho^\prime}(\tau)\|_{H^m}^2 d\tau\right) \nonumber \\
	& \leq & C_3 \left(\|g\|_{H_{\rho;\theta}^m}^2 + \int_0^t \|f(\tau)\|_{H_{\rho;\theta}^m}^2 d\tau\right), \quad t \in [0,T].
\end{eqnarray}
The conclusion here is: if we take the data $f\in C\left([0,T];H_{\rho;\theta}^m(\R)\right)$ and $g\in H_{\rho;\theta}^m(\R)$, for some $m\in\R$ and $\rho>0$, then we find a solution $u$ to the Cauchy problem associated with the operator $P$ with initial data $f$ and $g$ which satisfies
$$
u \in C\left([0,T];H_{\tilde{\rho};\theta}^m(\R)\right), \quad \tilde{\rho} < \rho.
$$

To prove the uniqueness of the solution, let us consider $u_1, u_2 \in C\left([0,T];H_{\tilde{\rho};\theta}^m(\R)\right)$ such that
$$
\left\lbrace \begin{array}{l}
	Pu_j=f \\ 
	u_j(0)=g
\end{array}  \right., \quad j=1,2.
$$
By shrinking $\rho^\prime$ if necessary, we can find new parameters $M_{p-1}^\prime,...,M_1^\prime>0$, $K^\prime>0$ and $h_0^\prime>0$  in order to apply Proposition \ref{proposition_auxiliar_energy} again and obtain that the Cauchy problem associated with the conjugated operator
$$
P_{\Lambda^\prime,K^\prime,\rho^\prime} := Q_{\Lambda^\prime,K^\prime,\rho^\prime}\circ(iP)\circ Q_{\Lambda^\prime,K^\prime,\rho^\prime}^{-1}
$$
is well-posed in $H^m(\R)$, where $Q_{\Lambda^\prime,K^\prime,\rho^\prime}$ comes from the new choice of parameters. Proposition \ref{proposition_auxiliar_energy} allows us to conclude that $
Q_{\Lambda^\prime,K^\prime,\rho^\prime}f, \ Q_{\Lambda^\prime,K^\prime,\rho^\prime}g, \ Q_{\Lambda^\prime,K^\prime,\rho^\prime} u_j \in H^m(\R)$ and satisfy
$$
\left\lbrace \begin{array}{l}
	P_{\Lambda^\prime,K^\prime,\rho^\prime}Q_{\Lambda^\prime,K^\prime,\rho^\prime}u_j=Q_{\Lambda^\prime,K^\prime,\rho^\prime}f \\ 
	Q_{\Lambda^\prime,K^\prime,\rho^\prime} u_j(0)=Q_{\Lambda^\prime,K^\prime,\rho^\prime}g
\end{array}  \right., \quad j=1,2,
$$
hence $Q_{\Lambda^\prime,K^\prime,\rho^\prime}u_1=Q_{\Lambda^\prime,K^\prime,\rho^\prime}u_2$, which implies that $u_1=u_2$.

\begin{remark}\label{bramble_scramble}
	In this final remark we observe that we could prove a slightly more general version of Theorem \ref{maintheorem1}, which is a kind of sufficiency counterpart for the main theorem in \cite{AACpevolGevreynec}. More into the point, we may consider the following hypotheses on the decay of the coefficients $a_{p-j}(t,x), j=1,\ldots,p-1$:
	\begin{equation}\label{the_panthom_of_the_opera}
		|\partial^{\beta}_{x} a_{p-j}(t,x)| \leq C_{p-j}^{\beta+1} \beta!^{\theta_0} \langle x \rangle^{-\sigma_{p_j}-\beta},
	\end{equation}
	for some $1 < \theta_0$ and $\sigma_{p-j}\in (0,1).$ Then the condition $\theta_0 \leq \theta < \frac{1}{(p-1)(1-\sigma)}$ for $\mathcal{H}^{\infty}_{\theta}(\R)$ well-posedness of \eqref{cauchy_problem_gevrey} would be:
	\begin{equation}\label{scars_of_time}
	\Xi < \frac{1}{\theta},
	\end{equation}
	where $\Xi$ is given by \eqref{equation_thesis_main_theorem}. In the next lines we shall explain how to obtain this improvement using the ideas developed in this paper.
	
	To prove the sufficiency of \eqref{scars_of_time} under \eqref{the_panthom_of_the_opera} we just need to consider a slightly different change of variables (cf. Section \ref{sec:changeofvariable}):
	$$
	\lambda_{p-j}(x,\xi) = M_{p-j} w(\xi h^{-1}) \langle \xi \rangle^{-(j-1)}_{h} \int_0^x \langle y \rangle^{-\sigma_{p-j}} \psi \left( \frac{\langle y \rangle}{\langle \xi \rangle^{p-1}_h}\right) dy, \quad j =1, \ldots, p-1,
	$$
	and 
	$$
	\Lambda_{K, \rho'}(t,\xi) = K(T-t)\langle \xi \rangle^{\Xi}_h + \rho'\langle \xi \rangle^{\frac{1}{\theta}}_{h}.
	$$
	Arguing as in Lemma \ref{estimatesforlambda} one gets that 
	$$
	\Lambda = \sum_{j = 1}^{p-1} \lambda_{p-j} \in \textrm{\textbf{S}}^{\Xi}_{\mu}(\R^2) \cap \textrm{\textbf{SG}}_\mu^{0,\Theta}(\R^2),
	$$
	where 
	$$
	\Theta := \max_{1 \leq j \leq p-1} \left\{ (1-\sigma_{p-j}) - \frac{j-1}{p-1} \right\}.
	$$
	Thus, we still have that $e^{\rho'\langle \xi \rangle^{\frac{1}{\theta}}_{h}}$ is the leading part of the operator
	$$
	Q_{\Lambda,K,\rho^\prime}(t,x,D)=e^{\Lambda_{K,\rho^\prime}}(t,D)\circ e^\Lambda(x,D).
	$$
	So, due to the hypotheses $\Xi < \frac{1}{\theta}$, $1 < \theta_0 \leq \theta$ and \eqref{the_panthom_of_the_opera}, one can still run the change of variable argument and conclude $\mathcal{H}^{\infty}_{\theta}(\R)$ well-posedness for \eqref{cauchy_problem_gevrey}.
	
	Of course the proof of the result described in this remark is even more technical than the proof of Theorem \ref{maintheorem1}. For this reason we decided to present a simpler result and leave the details of this more general version to the interested reader. 
\end{remark}

\textbf{Acknowledgements.} The authors wish to thank the referee for his/her precious comments and for suggesting new possible developments for our research on the equations treated in the paper.




\end{document}